\documentclass[11pt]{article}
\usepackage{amsmath}
\usepackage{amsfonts}
\usepackage{amssymb}
\usepackage{amsthm}
\usepackage[margin=1.3 in]{geometry}

\theoremstyle{remark}
\newtheorem{theorem}{Theorem}[section]
\newtheorem{lemma}{Lemma}[section]
\newtheorem{prop}{Proposition}[section]

\newtheorem{example}{Example}[section]
\newtheorem{rmrk}{Remark}[section]
\theoremstyle{definition}
\newtheorem{definition}{Definition}

\newcommand{\hadaprod}{\circ}
\newcommand{\regprod}{*}

\newcommand{\superclass}[1]{\left[#1\right]}
\newcommand{\aut}{\operatorname{Aut}}
\newcommand{\irr}{\operatorname{Irr}}

\newcommand{\kleingenone}{a}
\newcommand{\kleingentwo}{b}
\newcommand{\kleingenthree}{c}
\newcommand{\cpgen}{\omega}
\newcommand{\kleingenonechar}{\psi_a}
\newcommand{\kleingentwochar}{\psi_b}
\newcommand{\kleingenthreechar}{\psi_c}
\newcommand{\kleingenxchar}{\psi_x}
\newcommand{\kleingenychar}{\psi_y}
\newcommand{\cpgenchar}{\chi}
\newcommand{\rootofone}{\rho}

\newcommand{\Mset}[2]{M(#1,#2)}
\newcommand{\Msum}[2]{\widehat{M(#1,#2)}}

\newcommand{\divisor}{D}

\author{Alexander Lang}
\title{An Enumeration of the Supercharacter Theories of $C_p \times C_2\times C_2$ for Prime $p$}

\begin{document}
\maketitle

\begin{abstract}
The supercharacter theories of $C_p \times C_2 \times C_2$ were classified in the language of Schur rings by Evdokimov, Kov\'acs, and Ponomarenko in \cite{schurityAbelian}. It was shown that every nontrivial supercharacter theory of $C_p \times C_2 \times C_2$ can be constructed as a wedge product, a direct product, or is generated by automorphisms. We use this classification to give a precise count of the distinct supercharacter theories of $C_p\times C_2\times C_2$ and describe when a supercharacter theory can be constructed by more than one method. We also present an alternative proof of the classification using the language of supercharacter theories.
\end{abstract}

\section{Introduction}
A supercharacter theory for a finite group $G$ is a pair of set partitions, one of $G$ and one of the set of irreducible characters $\irr(G)$, satisfying a small set of conditions. They correspond to subalgebras of the group algebra $Z(\mathbb{C}G)$ satisfying some special properties. They were first defined by Diaconis and Isaacs \cite{DiaconisIsaacs} and were used to assist in understanding the groups $\text{UT}_n(q)$ of unimodular upper triangular matrices over the finite field $\mathbb{F}_q$. Since then, there has been research connecting supercharacter theories to Hopf algebras and number theory by examining specific classes of supercharacter theories.

One natural question that arises is what are all the possible supercharacter theories for a given group $G$? One important family of groups for which the answer to that question is known is the cyclic groups as described in \cite{cyclicSchur} and \cite{cyclicSchurII}. In this case there are three general methods of constructing supercharacter theories which suffice to construct all possible nontrivial supercharacter theories. The first method generates the supercharacter theory as the orbits of the action of a subgroup of the automorphisms of $G$. The other two methods build the supercharacter theory from the supercharacter theories of smaller groups. First, if $G$ can be expressed as a direct product of two subgroups, we can form the direct product of any pair of supercharacter theories of the two subgroups of $G$. Second, we can use what are known as wedge products. We will only describe a special case as it is sufficient for our discussion, although the full generality is required for cyclic groups. For any normal subgroup $N$ of $G$, the wedge product combines any supercharacter theory for $N$ and any supercharacter theory for $G/N$ to form a supercharacter theory for $G$. It is possible that a supercharacter theory can be constructed by more than one of these methods, or by the same method in different ways. Also for general Abelian groups, it is known that these three methods are not sufficient. In particular there are nontrivial supercharacter theories for $p$-groups which cannot be constructed using these methods.

 It was shown in the proof of Theorem 1.5 in \cite{schurityAbelian} that these three methods are also sufficient to construct every nontrivial supercharacter theory of $C_p \times C_2 \times C_2$ where $p$ is prime. Using this classification, it was proved that $C_p \times C_2 \times C_2$ is a Schur group for any prime $p$. We use this classification to determine the total number of distinct supercharacter theories of $C_p \times C_2 \times C_2$ for $p$ odd, and further which supercharacter theories can be constructed by each of the methods, and which ones can be constructed in more than one way. Many of these supercharacter theories are isomorphic, however we will consider them distinct if the partitions of $C_p\times C_2\times C_2$ are distinct regardless of whether they are isomorphic or not. In Sections 3 through 7, we present an alternative proof of the classification using the language and techniques of supercharacter theories. Our proof often uses more elementary and detailed methods than the one presented in \cite{schurityAbelian}, however it is also much longer. A key element of our argument is the utilization of the properties of the sums of roots of unity which occur when the supercharacters are evaluated on superclasses. These techniques may potentially allow the classification to be extended in different directions, particularly to supercharacter theories of nonAbelian groups. The proof presented in Theorem 1.5 in \cite{schurityAbelian} is recommended for readers familiar with the language of Schur rings.

\section{Preliminaries}

We will use the following notation. We will denote the cyclic group of order $n$ by $C_n$.  As we will discuss $(C_2)^3$ separately, let $p$ be an odd prime. All groups will be written multiplicatively with identity $e$. If $H_1$ and $H_2$ are normal subgroups of $G$ such that $H_1 \cap H_2 =\{e\}$ and $\langle H_1,H_2 \rangle =G$ then $G \cong H_1 \times H_2$ and we call $H_1$ and $H_2$ a \textit{complementary pair}. Given a group $G$ we will let $\irr(G)$ represent the set of irreducible characters of $G$ over $\mathbb{C}$. We shall denote the character of the trivial representation by $triv$. For $K\subseteq G$ define 
\begin{align}
&K^{(i)}=\{k^i|k\in K\},
\\
&\widehat{K}=\sum_{g \in K}{g} \in \mathbb{C}G.
\end{align}

We begin with the definition of a \textit{supercharacter theory} for a finite group based on the original description given by Diaconis and Isaacs in \cite{DiaconisIsaacs}:
\begin{definition}
\cite{DiaconisIsaacs} Given a finite group $G$, a \textit{supercharacter theory} for $G$ is a pair $(\mathcal{X},\mathcal{K})$ where $\mathcal{X}$ is a partition of $\irr(G)$ and $\mathcal{K}$ is a partition of the set of conjugacy classes of $G$ satisfying the following conditions:

\begin{enumerate}
\item $\{e\}$ is an element of $\mathcal{K}$, and $\{\operatorname{triv}\}$ is an element of  $\mathcal{X}$,

\item $|\mathcal{K}|=|\mathcal{X}|$,

\item for all $X\in \mathcal{X}$ and all $K\in \mathcal{K}$ the class functions $\displaystyle \sigma_X=\sum_{\chi\in X}{\chi(e)\chi}$ satisfy $\sigma_X(g)=\sigma_X(h)$ for all $g,h\in K$.
\end{enumerate}

The elements of $\mathcal{K}$ are called \textit{superclasses} and the $\sigma_X$, for all $X \in \mathcal{X}$ are called \textit{supercharacters}. We shall denote the superclass containing the element $g$ by $\superclass{g}_{\mathcal{K}}$ or by $\superclass{g}$ when $\mathcal{K}$ is clear from context. We shall denote the $X\in \mathcal{X}$ containing the irreducible character $\chi$ by $\superclass{\chi}_{\mathcal{X}}$ or $\superclass{\chi}$.
\end{definition}

We will say that a supercharacter theory $(\mathcal{X}',\mathcal{K}')$ is a \textit{refinement} of the supercharacter theory $(\mathcal{X},\mathcal{K})$ if $\mathcal{X}'$ is a refinement of $\mathcal{X}$ as partitions or equivalently $\mathcal{K}'$ is a refinement of $\mathcal{K}$.
\begin{prop}
\cite[Th. 2.2]{DiaconisIsaacs} If $(\mathcal{X},\mathcal{K}_1)$ and $(\mathcal{X},\mathcal{K}_2)$ are supercharacter theories for $G$, then $\mathcal{K}_1=\mathcal{K}_2$. Similarly if $(\mathcal{X}_1,\mathcal{K})$ and $(\mathcal{X}_2,\mathcal{K})$ are supercharacter theories for $G$ then $\mathcal{X}_1=\mathcal{X}_2$.
\end{prop}
\begin{lemma}
\cite[Lemma 6.1(a)]{Hendthesis}
If $(\mathcal{X},\mathcal{K})$ is a supercharacter theory for $G$ and $K\in \mathcal{K}$, then the subgroup of $G$ generated by $K$ is a union of superclasses.
\end{lemma}

We will require the use of an alternative description of a supercharacter theory based on the bijection between the set of supercharacter theories for a finite group $G$ and the set of Schur rings of $G$ which are contained in $Z(\mathbb{C}G)$ \cite{Hend}. For more information, see \cite{Wielandtbook}.

\begin{definition}
The \textit{Hadamard product} $\hadaprod$ is defined on $\mathbb{C}G$ by 
\begin{equation}
\displaystyle \left(\sum_{g\in G}{a_g g}\right)\hadaprod \left(\sum_{g\in G}{b_g g}\right)=\sum_{g\in G}{(a_g b_g) g}.
\end{equation}
\end{definition}
Note that $(\mathbb{C}G,\hadaprod)$ is a commutative associative algebra, with identity $\widehat{G}$. When necessary, we will denote the usual product on $\mathbb{C}G$ by $\regprod$, to distinguish it from the Hadamard product.

\begin{prop}
\label{algebraDef}
\cite[Prop. 2.4]{Hend} For a finite group $G$ there is a bijective correspondence between supercharacter theories $(\mathcal{X},\mathcal{K})$  and $\mathbb{C}$-linear subspaces $A$ of $Z(\mathbb{C}G)$ containing $e$ and $\widehat{G}$ which are closed under the operations $\regprod$ and $\hadaprod$.
\end{prop}

Because of the above bijection we shall also refer to such algebras as supercharacter theories. Given such an algebra $A$ we shall denote the corresponding partitions of $\irr(G)$ and $G$ by $\mathcal{X}_A$ and $\mathcal{K}_A$ respectively. 
\begin{rmrk}
Note that for $H \subseteq G$, $\widehat{H} \in A$ is equivalent to $H$ being a union of superclasses. Also for a supercharacter $\theta$ and an irreducible character $\chi$ the inner product $\langle \theta,\chi \rangle\not= 0$ is equivalent to $\chi$ a summand of $\theta$.
\end{rmrk}

We recall the following, for more details see \cite{DiaconisIsaacs}. We note that if $N$ is a normal subgroup of $G$ with a supercharacter theory $A$, and $\widehat{N} \in A$ then $\mathbb{C}N \cap A$ is a supercharacter theory for $N$. We shall call such a theory the \textit{restriction} of $A$ to $N$, denoted by $A|_N$. We also observe that $A|_N$ is the supercharacter theory of $N$ defined by the superclasses $\{K\in \mathcal{K}_A|K\subset N\}$. The supercharacter theory $A=Z(\mathbb{C}G)$ is called the \textit{minimal} supercharacter theory. The supercharacter theory with $\mathcal{K}=\{\{e\},G\setminus \{e\}\}$ is called the \textit{maximal} supercharacter theory, and we will also refer to it as the \textit{trivial} supercharacter theory.

We recall the following methods of constructing supercharacter theories. For cyclic groups, these three methods are sufficient to construct all nontrivial supercharacter theories. We note that it is sometimes possible to construct a given supercharacter theory using more than one of the following constructions.
\begin{prop}
\cite{DiaconisIsaacs} A subgroup $H$ of $\aut(G)$ acts on both the set of conjugacy classes of $G$ and $\irr(G)$. Letting $\mathcal{K}$ be the set of orbits of the action on the conjugacy classes and $\mathcal{X}$ be the orbits of the action on $\irr(G)$ yields a supercharacter theory.
\end{prop}

We will say that such a supercharacter theory $(\mathcal{X},\mathcal{K})$ is \textit{generated by} $H$, or \textit{generated by automorphisms}.
\begin{prop}
\cite[Prop 8.1]{Hend} Suppose that $(\mathcal{X}_G,\mathcal{K}_G)$ is a supercharacter theory for $G$ and $(\mathcal{X}_H,\mathcal{K}_H)$ is a supercharacter theory for $H$. Then there is a supercharacter theory for $G\times H$ in which the superclasses are given by $\{K \times L |K\in \mathcal{K}_G,L\in \mathcal{K}_H\}$ and the supercharacters are given by $\{\phi\times \sigma |\phi \in \mathcal{X}_G,\sigma \in \mathcal{X}_H\}$.
\end{prop}

We will often refer to a supercharacter theory constructed in this way as the \textit{direct product} of the supercharacter theories $(\mathcal{X}_G,\mathcal{K}_G)$ and $(\mathcal{X}_H,\mathcal{K}_H)$.
\begin{prop}
\cite[Th. 4.2]{Hend} Let $N$ be a normal subgroup of $G$, and let $\pi: G \rightarrow G/N$ be the natural quotient map. Suppose that $(\mathcal{X}_N,\mathcal{K}_N)$ is a supercharacter theory for $N$, and $(\mathcal{X}_{G/N},\mathcal{K}_{G/N})$ is a supercharacter theory for $G/N$. Then there is a supercharacter theory for $G$ in which the set of superclasses is
\begin{equation}
\mathcal{K}_N \cup \{\pi^{-1}(K)| K \in \mathcal{K}_{G/N}\setminus \{\{e\}\}\}.
\end{equation}
For $\psi \in \irr(N)$ let $\irr(G|\psi)$ be the set of $\chi \in \irr(G)$ such that $\langle \chi|_N , \psi \rangle > 0$. The corresponding partition of $\irr(G)$ is given by 
\begin{equation}
\mathcal{X}_{G/N} \cup \{ \bigcup_{\psi \in X} \irr(G|\psi)|X\in \mathcal{X}_N, X\not=\{\operatorname{triv}\}\}.
\end{equation}

\end{prop}
We will call such a supercharacter theory the \textit{wedge product} of the supercharacter theory for $N$ and the supercharacter theory of $G/N$, to agree with the conventions in \cite{cyclicSchur} and \cite{cyclicSchurII}.

We note that the above construction is a special case of a more general method, see \cite{Hend}. When classifying the supercharacter theories of cyclic groups the full generality is necessary, but in our case this version will suffice.

\section{Structure of the Main Argument}
We are now ready to present the classification given in the proof of Theorem 1.5 in \cite{schurityAbelian} in our current terminology:
\begin{theorem}\label{sufficient}
\cite{schurityAbelian}
Let $p$ be prime, $G=C_p\times C_2 \times C_2$. Then every nontrivial supercharacter theory $A$ of $G$ is at least one of the following:
\begin{enumerate}
\item generated by automorphisms,
\item the direct product of supercharacter theories for a pair of complementary subgroups $H_1,H_2 \leq G$,
\item the wedge product of supercharacter theories for $H\leq G$ and $G/H$.
\end{enumerate}
\end{theorem}
We will now outline the structure of our version of a proof. The following theorem is a key part of our argument.
\begin{theorem}
\label{Wielandt}
\cite[Th. 25.4]{Wielandtbook} If $G$ is an Abelian group of composite order and there exists a prime $p$ such that the $p$-Sylow subgroup of $G$ is nontrivial and cyclic, then for every nontrivial supercharacter theory $A$ of $G$ there exists a proper nontrivial subgroup $H$ such that $\widehat{H} \in A$.
\end{theorem}
It is clear that this theorem applies  to $C_p\times C_2\times C_2$ when $p$ is an odd prime. Our classification will be split into four cases based on what collection of proper nontrivial subgroups are unions of superclasses, and by this theorem we know that this collection is nonempty.  Every proper nontrivial subgroup of $C_p\times C_2\times C_2$ is isomorphic to $C_2$, $C_2 \times C_2$, $C_p$, or $C_p \times C_2$. There are three isomorphic copies of $C_2$ and $C_p\times C_2$ respectively. There is one copy of $C_2\times C_2$ and one copy of $C_p$.

We see that there are a large number of different possibilities for which subgroups are unions of superclasses. However, we can reduce the number of cases which must be considered in the following way. Recall that if $G$ is Abelian then $\irr(G) \cong G$ as groups, although the isomorphism is non-canonical. In \cite{Hendconstruction} the notion of a \textit{dual supercharacter theory} is introduced using the canonical isomorphism $G \cong \irr(\irr(G))$, and it is shown that every supercharacter theory $A$ of an Abelian group $G$ yields a unique supercharacter theory $B$ of $\irr(G)$ where $\mathcal{K}_B=\mathcal{X}_A$. $B$ is not in general isomorphic to $A$. We will make frequent use of the following lemma, for convenience we present a proof:
\begin{lemma}
\label{dual}
\cite[Lemma 11.1]{Hendconstruction} Let $A$ be a supercharacter theory for $G$ and let $B$ be its dual supercharacter theory for $\irr(G)$. If $H\leq G$ and $\widehat{H} \in A$ then there exists $N \leq \irr(G)$ such that $N\cong G/H$ and  $\widehat{N} \in B$. Further this relation is inclusion reversing: if  $\widehat{H'} \in A$ and $H\leq H'$ then there exists $N' \leq N$ such that $N'\cong G/H'$ and $\widehat{N'} \in B$.
\end{lemma}
\begin{proof}
Consider the dimension three supercharacter theory of $G$ with superclasses $\{e\}, H\setminus \{e\}, G\setminus H$. The corresponding partition of $\irr(G)$ is $\{\operatorname{triv}\}$, $N\setminus \{\operatorname{triv}\}$, and $\irr(G)\setminus N$ where $N$ is the subgroup of $\irr(G)$ satisfying
\begin{equation}
\bigcap_{\chi \in N} \ker \chi =H.
\end{equation}
Then by considering the superclasses, we see that $A$ is a refinement of this theory. Hence $N$ is a union of elements of $\mathcal{X}_A$ as desired. It is clear that if $H \leq H'$, then $N' \leq N$, as every $\chi$ which contains $H'$ in its kernel also contains $H$ in its kernel, so we are done.
\end{proof}
We also recall the following part of Corollary 11.6 of \cite{Hendconstruction}:
\begin{lemma}
\cite[Cor. 11.6]{Hendconstruction} A supercharacter theory $A$ of $G$ can be constructed as a wedge product iff the dual of $A$ can be constructed as a wedge product.
\end{lemma}
See \cite{Hendconstruction} for details on the dual supercharacter theory, and \cite{SchurGalois} for duality in the Schur ring setting.

We note that $G=C_p \times C_2 \times C_2$ has the convenient property that $H,H' \leq G$ with $H \cong H'$ implies $G/H \cong G/H'$. It is now clear that if we describe the collection of all supercharacter theories $A$ of $G=C_p\times C_2 \times C_2$ such that $H \leq G$ and $\widehat{H}\in A$ then we also have a description of the collection of all supercharacter theories $B$ of $G$ with $\widehat{N} \in B$ where $N \cong G/H$ by considering the dual supercharacter theory and using any fixed isomorphism $G \rightarrow \irr(G)$. We first list all possible sets of proper nontrivial subgroups which can occur as unions of superclasses, excluding those sets which contain a complementary pair of subgroups as we will use separate arguments for them. It is a simple exercise to verify that this list is complete:
\begin{enumerate}
\item $C_2$
\item $C_2 \times C_2$
\item one $C_2$ and $C_2 \times C_2$
\item all three $C_2$ and $C_2 \times C_2$
\item $C_p$
\item $C_p \times C_2$
\item $C_p$ and $C_p \times C_2$
\item $C_p$ and all three $C_p \times C_2$
\item $C_2$ and the $C_p \times C_2$ containing it
\item $C_2$, $C_p$, and the $C_p \times C_2$ containing them
\item $C_2$, the $C_p \times C_2$ containing it, and $C_2\times C_2$.
\end{enumerate}
We now match these cases into pairs as in Lemma \ref{dual}:
\begin{enumerate}
\item[(i)] $
C_p\times C_2 \leftrightarrow C_2
$

\item[(ii)] $
C_p \leftrightarrow C_2\times C_2
$

\item[(iii)] $
C_p\times C_2 \text{ and } C_p \leftrightarrow C_2\times C_2 \text{ and one } C_2
$

\item[(iv)] $
C_p \text{ and all three } C_p \times C_2 \leftrightarrow C_2\times C_2 \text{ and all three } C_2
$

\item[(v)] $
C_2 \text{ and the } C_p\times C_2 \text{ which contains it } \leftrightarrow C_2 \text{ and the } C_p\times C_2 \text{ which contains it }
$

\item[(vi)] $
C_p, C_2, \text{ and the } C_p\times C_2 \text{ containing them } \leftrightarrow C_2, \text{ the } C_p\times C_2 \text{ containing it, and } C_2\times C_2.
$
\end{enumerate}
By the above argument, we only need to consider one of the cases in each corresponding pair. We shall handle Cases (ii), (iii), (iv), and (vi) in one argument, by assuming that $\widehat{C_p} \in A$. The remaining two cases, (i) and (v), we shall argue by assuming that one copy of $C_p\times C_2$ satisfies $\widehat{C_p\times C_2} \in A$, and that the $C_2$ contained in it is the only other proper nontrivial subgroup $H$ which may satisfy $\widehat{H}\in A$.

We shall organize our argument as follows. (i) and (v) will be considered in Case 1. (ii), (iii), (iv), (vi) will be in Case 2 where we will assume $\widehat{C_p} \in A$ and $\widehat{C_2\times C_2} \notin A$. The situation where there exists a complementary pair $H_1$ and $H_2$ with $\widehat{H_1},\widehat{H_2} \in A$ will be considered in Cases 3 and 4. We will conclude our proof in section \ref{p2}, where we consider the case when $G=(C_2)^3$.

We introduce some more notation. Let $C_p=\langle \cpgen \rangle$ and $C_2\times C_2=\{ e,\kleingenone,\kleingentwo,\kleingenthree \}$. We fix a primitive $p$th root of unity $\rootofone$. We define $\cpgenchar \in \irr(C_p\times C_2\times C_2)$ to be the irreducible character defined by $\cpgenchar(\cpgen)=\rootofone$, $\cpgenchar(\kleingenone)=\cpgenchar(\kleingentwo)=\cpgenchar(\kleingenthree)=1$. We let $\kleingenonechar$, $\kleingentwochar$, and $\kleingenthreechar$ be the irreducible characters defined by $\psi_x(\cpgen)=1$, $\psi_x(x)=1$, $\psi_x(y)=-1$ for all $x,y \in \{\kleingenone,\kleingentwo,\kleingenthree\}$ with $x,y$ distinct. We will let $A$ denote a supercharacter theory of $C_p\times (C_2\times C_2)$. If $G=H_1\times H_2$ then for $K\subset G$ and $h\in H_2$ we define $\Mset{K}{h}=\{g\in H_1|gh\in K\}$. Hence 
\begin{equation}
K=\bigcup_{h\in H_2}\Mset{K}{h}h.
\end{equation}
See Example \ref{exwithM}. Similarly if $\irr(G)=H_1\times H_2$, and $\chi \in H_2$, then for a supercharacter $\theta=\sigma_X$ we will define $\Mset{\theta}{\chi}=\{\psi\in H_1 |\psi \chi \in X \}$. Note that $\psi \in \Mset{\theta}{\chi}$ is equivalent to $\langle \theta,\psi \chi \rangle\not= 0$. Finally, we note the following observation about $\rootofone$:
\begin{rmrk}\label{keyunity}
Since $1+t+t^2+\ldots+t^{p-1}$ is the minimal polynomial of $\rootofone$ over $\mathbb{Q}$ we have that if $\ell \not \equiv 0 \pmod p$ and $\displaystyle \sum_{i=1}^{p-1}{\rootofone^{i\ell}}z_i \in \mathbb{Z}$ where $z_i \in \mathbb{Q}$, then for all $i,j \not \equiv 0 \pmod p$ $z_i=z_j$.
\end{rmrk}
\section{Case 1}

We begin with the case where $\widehat{C_p\times C_2} \in A$ and the $C_2$ subgroup contained in this copy of $C_p\times C_2$ is the only other nontrivial proper subgroup $H$ which may have $\widehat{H}\in A$. This corresponds to (i) and (v) above. Without loss of generality we let $\widehat{\langle \cpgen \rangle \times \langle \kleingenone \rangle} \in A$ and $\langle \kleingenone \rangle$ the only other proper nontrivial subgroup $H$ which may have $\widehat{H}\in A$. By Lemma \ref{dual} this is equivalent to $\{\kleingenonechar\}\in\mathcal{X}_A$ and $\langle \cpgenchar ,\kleingenonechar\rangle$ is the only other proper nontrivial subgroup of $\irr(G)$ which may be a union of elements of $\mathcal{X}_A$. $\{\kleingentwo\}$ is not a superclass so there exists $x \not= \kleingentwo$ such that $x \in \superclass{\kleingentwo}$. Since $\widehat{\langle\cpgen\rangle \times \langle \kleingenone \rangle} \in A$ $x \not= \cpgen^\ell, \cpgen^\ell\kleingenone$ for any $\ell$. Therefore $x=\kleingentwo\cpgen^\ell$ or $x=\kleingenthree\cpgen^\ell$ for some $\ell$. Similarly $\kleingentwo\cpgen^k\in\superclass{\kleingenthree}$ or $\kleingenthree\cpgen^k\in\superclass{\kleingenthree}$ for some $k$. $\{\kleingentwo,\kleingenthree\}$ a superclass implies that $\widehat{\langle \kleingentwo,\kleingenthree \rangle} \in A$ which contradicts our assumption. So we may choose $\ell,k\not \equiv 0 \pmod p$. Let $\theta \not= \kleingenonechar, \operatorname{triv}$ be a supercharacter. We will express $\irr(G)$ as $H_1\times H_2$ where $H_1=\langle \kleingenonechar,\kleingentwochar\rangle$ and $H_2=\langle \cpgenchar \rangle$. Hence 
\begin{equation}
\theta=\sum_{i=0}^{p-1}{\Msum{\theta}{\cpgenchar^i}\cpgenchar^i}.
\end{equation}
Since $\theta(\kleingentwo)=\theta(x)$ we have:
\begin{equation}
\sum_{i=0}^{p-1}{\Msum{\theta}{\cpgenchar^i}(\kleingentwo)} = \sum_{i=0}^{p-1}{\rootofone^{\ell i}\Msum{\theta}{\cpgenchar^i}(x)}
\end{equation}
\begin{equation}
\left(\sum_{i=0}^{p-1}{\Msum{\theta}{\cpgenchar^i}(\kleingentwo)}\right)-\Msum{\theta}{\cpgenchar^0}(x) = \sum_{i=1}^{p-1}{\rootofone^{\ell i}\Msum{\theta}{\cpgenchar^i}(x)}.
\end{equation}
The LHS of the equation above is an integer, therefore by Remark \ref{keyunity} for all $i,j \not \equiv 0 \pmod p$:
\begin{equation}
\Msum{\theta}{\cpgenchar^i}(x)=\Msum{\theta}{\cpgenchar^j}(x)
\end{equation}
\begin{equation}\label{case1maineq}
\left(\sum_{i=0}^{p-1}{\Msum{\theta}{\cpgenchar^i}(\kleingentwo)}\right)-\Msum{\theta}{\cpgenchar^0}(x) = -\Msum{\theta}{\cpgenchar}(x).
\end{equation}

\textbf{Case 1a:}

We begin with the situation where $\kleingentwo \cpgen^\ell \in \superclass{\kleingentwo}$ and $\kleingentwo \cpgen^k \in \superclass{\kleingenthree}$ for some $\ell, k \not \equiv 0 \pmod p$, and further $\kleingenthree \cpgen^u \notin \superclass{\kleingentwo}$ for any $u\not \equiv 0 \pmod p$ and $\kleingenthree \cpgen^v \notin \superclass{\kleingenthree}$ for any $v\not \equiv 0 \pmod p$. Note that for all $i,r$
\begin{equation}
\Msum{\theta}{\cpgenchar^i}(\kleingentwo)=\Msum{\theta}{\cpgenchar^i}(\kleingentwo\cpgen^r).
\end{equation}
Then we have:
\begin{equation}
\left(\sum_{i=0}^{p-1}{\Msum{\theta}{\cpgenchar^i}(\kleingentwo)}\right)-\Msum{\theta}{\operatorname{triv}}(\kleingentwo)=-\Msum{\theta}{\cpgenchar}(\kleingentwo)
\end{equation}
\begin{equation}
\sum_{i=1}^{p-1}{\Msum{\theta}{\cpgenchar^i}(\kleingentwo)}=-\Msum{\theta}{\cpgenchar}(\kleingentwo)
\end{equation}
\begin{equation}
(p-1)\Msum{\theta}{\cpgenchar}(\kleingentwo)=-\Msum{\theta}{\cpgenchar}(\kleingentwo)
\end{equation}
\begin{equation}
\Msum{\theta}{\cpgenchar}(\kleingentwo)=0.
\end{equation}
Hence for all $i\not \equiv 0 \pmod p$
\begin{equation}
\Msum{\theta}{\cpgenchar^i}(\kleingentwo)=0.
\end{equation}
Then we conclude that for all $r$
\begin{equation}
\theta(\kleingentwo \cpgen^r)=\sum_{i=0}^{p-1}{\rootofone^{ri}\Msum{\theta}{\cpgenchar^i}(\kleingentwo)}=\Msum{\theta}{\operatorname{triv}}(\kleingentwo).
\end{equation}
So $\theta$ is constant on $\{\kleingentwo,\kleingentwo \cpgen, \ldots,\kleingentwo\cpgen^{p-1}\}$. For all $r$, $\kleingenonechar(\kleingentwo \cpgen^r)=-1$, so $\kleingenonechar$ is also constant on $\{\kleingentwo,\kleingentwo \cpgen, \ldots,\kleingentwo\cpgen^{p-1}\}$. Therefore every supercharacter is constant on the set $\{\kleingentwo,\kleingentwo \cpgen, \ldots,\kleingentwo\cpgen^{p-1}\}$, hence it is a subset of a superclass. Since $\kleingentwo \cpgen^\ell \in \superclass{\kleingentwo}$ and $\kleingentwo \cpgen^k \in \superclass{\kleingenthree}$, $\{\kleingenthree,\kleingentwo,\kleingentwo \cpgen, \ldots,\kleingentwo\cpgen^{p-1}\}$ is a subset of a superclass. By our Case 1a assumption, this must be a superclass. This would imply
\begin{equation}
\left(\kleingenthree+\sum_{i=0}^{p-1}{\kleingentwo \cpgen^i}\right)^2=e+2\sum_{i=0}^{p-1}{\kleingenone \cpgen^i}+p\sum_{i=0}^{p-1}{\cpgen^i}\in A.
\end{equation}
Since $p\not=2$, this implies $\widehat{\langle \cpgen \rangle}\in A$, which contradicts our assumption for Case 1. Similarly, we have a contradiction in the case of $\kleingenthree \cpgen^\ell \in \superclass{\kleingentwo}$, $\kleingenthree \cpgen^k \in \superclass{\kleingenthree}$ for some $\ell, k\not \equiv 0 \pmod p$, $\kleingentwo \cpgen^u \notin \superclass{\kleingentwo}$ for any $u\not \equiv 0 \pmod p$ and $\kleingentwo \cpgen^v \notin \superclass{\kleingenthree}$ for any $v\not \equiv 0 \pmod p$. 

\textbf{Case 1b:}

By the above, we can choose $x,y \in \{\kleingentwo,\kleingenthree \}$ so that $x\not= y$, $x \cpgen^\ell \in \superclass{\kleingentwo}$, $\ell\not \equiv 0 \pmod p$ and $y \cpgen^k \in \superclass{\kleingenthree}$, $k\not \equiv 0 \pmod p$.

Recall that for all $i,j \not \equiv 0 \pmod p$ $\Msum{\theta}{\cpgenchar^i}(\kleingentwo)=\Msum{\theta}{\cpgenchar^j}(\kleingentwo)$ and $\Msum{\theta}{\cpgenchar^i}(\kleingenthree)=\Msum{\theta}{\cpgenchar^j}(\kleingenthree)$. Then Equation (\ref{case1maineq})
becomes
\begin{equation}
(p-1)\Msum{\theta}{\cpgenchar}(\kleingentwo)+\Msum{\theta}{\operatorname{triv}}(\kleingentwo)-\Msum{\theta}{\operatorname{triv}}(x)=-\Msum{\theta}{\cpgenchar}(x).
\end{equation}
Similarly we have
\begin{equation}
(p-1)\Msum{\theta}{\cpgenchar}(\kleingenthree)+\Msum{\theta}{\operatorname{triv}}(\kleingenthree)-\Msum{\theta}{\operatorname{triv}}(y)=-\Msum{\theta}{\cpgenchar}(y).
\end{equation}

If $x=\kleingentwo$ and $y=\kleingenthree$, then $\Msum{\theta}{\cpgenchar}(\kleingentwo)=0$ and $\Msum{\theta}{\cpgenchar}(\kleingenthree)=0$. If $x=\kleingenthree$ and $y=\kleingentwo$ then
\begin{equation}\label{case1bv1eq}
(p-1)\Msum{\theta}{\cpgenchar}(\kleingentwo)+\Msum{\theta}{\operatorname{triv}}(\kleingentwo)-\Msum{\theta}{\operatorname{triv}}(\kleingenthree)+\Msum{\theta}{\cpgenchar}(\kleingenthree)=0,
\end{equation}
\begin{equation}\label{case1bv2eq}
(p-1)\Msum{\theta}{\cpgenchar}(\kleingenthree)+\Msum{\theta}{\operatorname{triv}}(\kleingenthree)-\Msum{\theta}{\operatorname{triv}}(\kleingentwo)+\Msum{\theta}{\cpgenchar}(\kleingentwo)=0.
\end{equation}
Adding these equations gives
\begin{equation}
p\Msum{\theta}{\cpgenchar}(\kleingentwo)+p\Msum{\theta}{\cpgenchar}(\kleingenthree)=0
\end{equation}
\begin{equation}\label{case1negeq}
\Msum{\theta}{\cpgenchar}(\kleingentwo)=-\Msum{\theta}{\cpgenchar}(\kleingenthree).
\end{equation}
We want to show that we cannot have $\Msum{\theta}{\operatorname{triv}} \in \{ \kleingentwochar, \kleingenthreechar \}$. Without loss of generality we assume that $\Msum{\theta}{\operatorname{triv}}=\kleingentwochar$ and $\Msum{\kleingenonechar \theta}{\operatorname{triv}}=\kleingenthreechar$. Then Equations (\ref{case1bv1eq}) and (\ref{case1bv2eq}) yield
\begin{align}
&(p-1)\Msum{\theta}{\cpgenchar}(\kleingentwo)+\Msum{\theta}{\cpgenchar}(\kleingenthree)=-2,
\\
&(p-1)\Msum{\theta}{\cpgenchar}(\kleingenthree)+\Msum{\theta}{\cpgenchar}(\kleingentwo)=2.
\end{align}
Hence by Equation (\ref{case1negeq}), $(p-2)\Msum{\theta}{\cpgenchar}(\kleingenthree)=2$. Since $\Msum{\theta}{\cpgenchar}(\kleingenthree)\in \mathbb{Z}$, we have $p=3$. Then $\Msum{\theta}{\cpgenchar}(\kleingenthree)=2$ implies $\Msum{\theta}{\cpgenchar}=\operatorname{triv}+\kleingenthreechar$. Then $\Msum{\theta}{\cpgenchar}(\kleingentwo)=0$ which  is a contradiction.
 
Hence we must have $\Msum{\theta}{\operatorname{triv}}=\kleingentwochar+\kleingenthreechar$ or $\Msum{\theta}{\operatorname{triv}}=0$, so Equations (\ref{case1bv1eq}) and (\ref{case1bv2eq}) yield
\begin{align}
&(p-1)\Msum{\theta}{\cpgenchar}(\kleingentwo)+\Msum{\theta}{\cpgenchar}(\kleingenthree)=0,
\\
&(p-1)\Msum{\theta}{\cpgenchar}(\kleingenthree)+\Msum{\theta}{\cpgenchar}(\kleingentwo)=0.
\end{align}
Subtracting gives
\begin{equation}
(p-2)\Msum{\theta}{\cpgenchar}(\kleingentwo)-(p-2)\Msum{\theta}{\cpgenchar}(\kleingenthree)=0
\end{equation}
\begin{equation}
\Msum{\theta}{\cpgenchar}(\kleingentwo)=\Msum{\theta}{\cpgenchar}(\kleingenthree).
\end{equation}
Using Equation (\ref{case1negeq}), we see that this implies $\Msum{\theta}{\cpgenchar}(\kleingentwo)=\Msum{\theta}{\cpgenchar}(\kleingenthree)=0$.
Further we see that $\Msum{\theta}{\operatorname{triv}}(\kleingentwo)=\Msum{\theta}{\operatorname{triv}}(\kleingenthree)=0$ as well, since $\Msum{\theta}{\operatorname{triv}}=0$ or $\Msum{\theta}{\operatorname{triv}}= \kleingentwochar+\kleingenthreechar$. Recalling that for all $i,j \not \equiv 0 \pmod p$ $\Msum{\theta}{\cpgenchar^i}(\kleingentwo)=\Msum{\theta}{\cpgenchar^j}(\kleingentwo)$ and $\Msum{\theta}{\cpgenchar^i}(\kleingenthree)=\Msum{\theta}{\cpgenchar^j}(\kleingenthree)$, we can now conclude that for any $r$
\begin{equation}
\theta(\cpgen^r \kleingentwo)=\sum_{i=0}^{p-1}{\rootofone^{ir}\Msum{\theta}{\cpgenchar^i}(\kleingentwo)}=0,
\end{equation}
\begin{equation}
\theta(\cpgen^r \kleingenthree)=\sum_{i=0}^{p-1}{\rootofone^{ir}\Msum{\theta}{\cpgenchar^i}(\kleingenthree)}=0.
\end{equation}
Since $\kleingenonechar(\cpgen^r \kleingentwo)=\kleingenonechar(\cpgen^r \kleingenthree)=-1$, we have that 
\begin{equation}
\{\kleingentwo, \cpgen\kleingentwo,\ldots,\cpgen^{p-1}\kleingentwo,\kleingenthree,\cpgen\kleingenthree,\ldots,\cpgen^{p-1}\kleingenthree\}
\end{equation}
is a subset of a superclass. Since $\widehat{\langle \cpgen \rangle\times \langle \kleingenone\rangle} \in A$, we see that the above must in fact be a superclass. Therefore this supercharacter theory is a wedge product of a supercharacter theory for $C_p \times C_2$ and a supercharacter theory for the quotient $(C_p\times C_2 \times C_2)/(C_p \times C_2)$. Since the quotient group is isomorphic to $C_2$, there is only one choice of a supercharacter theory for it. Hence this supercharacter theory is completely determined by the choice of a supercharacter theory for $C_p \times C_2$. This completes the classification in the case $\widehat{\langle \cpgen \rangle \times \langle \kleingenone \rangle} \in A$ and $\langle \kleingenone \rangle$ is the only other proper nontrivial subgroup $H$ which may have $\widehat{H}\in A$.

\section{Case 2}

We now consider the case where $\widehat{C_p} \in A$ and $\widehat{C_2\times C_2} \notin A$. By Lemma \ref{dual} this is equivalent to $\kleingenonechar+\kleingentwochar+\kleingenthreechar$ is a sum of supercharacters, and $\widehat{\langle \cpgenchar \rangle}$ is not a sum of supercharacters. 

\textbf{Case 2a:}

We begin by assuming there is a $\widehat{C_2} \in A$, and without loss of generality we let it be $\langle \kleingenone \rangle$ which implies by Lemma \ref{dual} that $\widehat{\langle \kleingenonechar,\cpgenchar \rangle}$ is a sum of supercharacters. Since $\langle \kleingenonechar,\cpgenchar \rangle\cap \{\kleingenonechar,\kleingentwochar,\kleingenthreechar\}=\{\kleingenonechar\}$, $\kleingenonechar$ is a supercharacter. Since $\widehat{C_2\times C_2} \notin A$, we know that $\{\kleingentwo\}, \{\kleingenthree\}$, and $\{\kleingentwo,\kleingenthree\}$ are not superclasses. Since $\widehat{\langle \cpgen, \kleingenone \rangle} \in A$, we have that $x\cpgen^\ell \in \superclass{\kleingentwo}$ for some $x \in \{\kleingentwo,\kleingenthree\}$ and $\ell\not \equiv 0 \pmod p$.

Let $\theta$ be a supercharacter which is a sum of irreducible characters in $\langle \cpgenchar, \kleingenonechar \rangle$. Since $\kleingenonechar$ is a supercharacter, $\Mset{\theta}{\operatorname{triv}}=0$. Then 
\begin{equation}
\theta(\kleingentwo)=\theta(x\cpgen^\ell)=\sum_{i=1}^{p-1}{\rootofone^{i\ell}\Msum{\theta}{\cpgenchar^i}(x)}.
\end{equation}
Since $\theta(\kleingentwo)\in \mathbb{Z}$, we conclude by Remark \ref{keyunity} that for all $i, j\not \equiv 0 \pmod p$,
\begin{equation}
\Msum{\theta}{\cpgenchar^i}(x)=\Msum{\theta}{\cpgenchar^j}(x).
\end{equation}
$\Mset{\theta}{\cpgenchar^i}=0$, $\operatorname{triv}$, $\kleingenonechar$, or $\operatorname{triv}+\kleingenonechar$, so $\Msum{\theta}{\cpgenchar^i}(x)=0$, $1$, or $-1$. If $\Msum{\theta}{\cpgenchar^i}(x)=1$, then $\Msum{\theta}{\cpgenchar^i}=\operatorname{triv}$. If $\Msum{\theta}{\cpgenchar^i}(x)=-1$, then $\Msum{\theta}{\cpgenchar^i}=\kleingenonechar$. If $\Msum{\theta}{\cpgenchar^i}(x)=0$ then $\Msum{\theta}{\cpgenchar^i}=0$ or $\Msum{\theta}{\cpgenchar^i}=\operatorname{triv} + \kleingenonechar$. If $\Msum{\theta}{\cpgenchar}=\operatorname{triv}$, then $\theta =\cpgenchar+\ldots +\cpgenchar^{p-1}$, which is a contradiction as $\cpgenchar+\ldots +\cpgenchar^{p-1}$ a sum of supercharacters is equivalent to $C_2\times C_2$ a union of superclasses. If $\Msum{\theta}{\cpgenchar}=\kleingenonechar$, then $\theta=\kleingenonechar(\cpgenchar+\ldots +\cpgenchar^{p-1})$, so again we have a contradiction. Therefore we conclude that for every $i$ either $\Msum{\theta}{\cpgenchar^i}=0$ or $\Msum{\theta}{\cpgenchar^i}=\operatorname{triv} + \kleingenonechar$. Hence for all $r=0,\ldots, p-1$ we have
\begin{equation}
\theta(\kleingentwo \cpgen^r)=\theta(\kleingenthree \cpgen^r)=0.
\end{equation}

Let $\tilde{\theta}$ be a supercharacter which is a sum of irreducible characters each of which is contained in $\{\kleingentwochar \cpgenchar, \kleingentwochar \cpgenchar^2, \ldots, \kleingentwochar \cpgenchar^{p-1}, \kleingenthreechar \cpgenchar, \kleingenthreechar \cpgenchar^2, \ldots, \kleingenthreechar \cpgenchar^{p-1}\}$. $\Mset{\tilde{\theta}}{\operatorname{triv}}=0$, and we have similarly:
\begin{equation}
\tilde{\theta}(b)=\tilde{\theta}(x\cpgen^\ell)=\sum_{i=1}^{p-1}{\rootofone^{i\ell}\Msum{\tilde{\theta}}{\cpgenchar^i}(x)}.
\end{equation}
$\tilde{\theta}(b) \in \mathbb{Z}$, so by Remark \ref{keyunity} for all $i,j$,
\begin{equation}
\Msum{\tilde{\theta}}{\cpgenchar^i}(x)=\Msum{\tilde{\theta}}{\cpgenchar^j}(x).
\end{equation}
If $\tilde{\theta}$ is a sum of irreducible characters contained in $\langle \cpgenchar, \kleingentwochar \rangle$, then $\widehat{\langle \cpgenchar, \kleingentwochar \rangle}$ is a sum of supercharacters which by Lemma \ref{dual} contradicts the assumption that $\{\kleingentwo\}$ isn't a superclass. Similarly $\tilde{\theta}$ is not a sum of characters contained in $\langle \cpgenchar, \kleingenthreechar \rangle$ because $\{\kleingenthree\}$ isn't a superclass. Therefore for all $i$ either $\Msum{\tilde{\theta}}{\cpgenchar^i}=0$ or $\Msum{\tilde{\theta}}{\cpgenchar^i}=\kleingentwochar+\kleingenthreechar$. Hence for all $r=0,\ldots, p-1$ we have
\begin{equation}
\tilde{\theta}(\kleingentwo \cpgen^r)=\tilde{\theta}(\kleingenthree \cpgen^r)=0.
\end{equation}

Recall that $\langle \cpgen \rangle$ a union of superclasses is equivalent to $\kleingenonechar+\kleingentwochar+\kleingenthreechar$ is a sum of supercharacters, and $\langle \cpgen, \kleingenone \rangle$ a union of superclasses is equivalent to $\kleingenonechar$ a supercharacter. Since neither $\langle \cpgen, \kleingentwo \rangle$ or $\langle \cpgen, \kleingenthree \rangle$ is a union of superclasses, neither $\kleingentwochar$ or $\kleingenthreechar$ is a supercharacter. Therefore $\kleingenonechar$ and $\kleingentwochar+\kleingenthreechar$ are both supercharacters.

We now are able to conclude that all supercharacters are constant on the set
\begin{equation}
\{\kleingentwo,\kleingentwo \cpgen,\ldots,\kleingentwo \cpgen^{p-1},\kleingenthree,\kleingenthree \cpgen,\ldots, \kleingenthree \cpgen^{p-1}\}
\end{equation}
and we conclude that it is a superclass. Since it is a $\langle \cpgen, \kleingenone \rangle$-coset, we conclude that $A$ is a wedge product  as in Case 1.

\textbf{Case 2b:}

We now assume that there is no $C_2$ with $\widehat{C_2} \in A$. We want to show that every superclass disjoint from $\langle \cpgen \rangle$ is a union of $\langle \cpgen \rangle$-cosets.
Since $C_2\times C_2 \notin A$, $\{\kleingenone,\kleingentwo,\kleingenthree\}$ is not a superclass, and by assumption $\{\kleingenone\},\{\kleingentwo\}$ and $\{\kleingenthree\}$ are not superclasses either. If $\{\kleingenone,\kleingentwo\}$ is a superclass, then $\widehat{\langle \kleingenone,\kleingentwo \rangle} \in A$ which contradicts $\{\kleingenone,\kleingentwo,\kleingenthree \}$ not a superclass. Similarly $\{\kleingenone,\kleingenthree\}$ and $\{\kleingentwo,\kleingenthree\}$ are not superclasses. Therefore we conclude that if $x \in \{\kleingenone,\kleingentwo,\kleingenthree\}$, then there exists $y \in \{\kleingenone,\kleingentwo,\kleingenthree\}$ and $\ell \not \equiv 0 \pmod p$ such that $y\cpgen^\ell \in \superclass{x}$. Let $\theta \not= \operatorname{triv}$ be a supercharacter which is not a summand of $\kleingenonechar+\kleingentwochar+\kleingenthreechar$. Then
\begin{equation}
\theta(x)=\theta(y\cpgen^\ell)
\end{equation}
\begin{equation}\label{thetaeqexpanded}
\sum_{i=1}^{p-1}{\Msum{\theta}{\cpgenchar^i}(x)}=\sum_{i=0}^{p-1}{\Msum{\theta}{\cpgenchar^i}(x)}=\sum_{i=0}^{p-1}{\rootofone^{\ell i}\Msum{\theta}{\cpgenchar^i}(y)}=\sum_{i=1}^{p-1}{\rootofone^{\ell i}\Msum{\theta}{\cpgenchar^i}(y)}.
\end{equation}
Since $\theta(x)\in \mathbb{Z}$, by Remark \ref{keyunity} we see that for all $i\not\equiv 0 \pmod p$ 
\begin{equation}\label{constoncp}
\theta(y\cpgen^i)=-\Msum{\theta}{\cpgenchar}(y).
\end{equation}
In particular, note that this implies $\theta$ is constant on $\{y\cpgen,y\cpgen^2,\ldots,y\cpgen^{p-1}\}$. For a superclass $K$ let $I_K=\{x\in C_2\times C_2 | \exists \ell \text{ such that } x\cpgen^\ell \in K\}$. Given superclasses $K,K'$ which are not subsets of $\langle \cpgen \rangle$, we claim that $I_K$ and $I_{K'}$ are either equal or disjoint. Suppose they are not disjoint. Then we have
\begin{equation}
\widehat{K}\widehat{\langle \cpgen \rangle}=\sum_{x\in K}{\sum_{i=0}^{p-1}{x\cpgen^i}} \in A
\end{equation}
since $\widehat{K'}$ must be a summand, this implies $I_{K'}\subset I_K$ . However, by considering $\widehat{K'}\widehat{\langle \cpgen \rangle}$ we see that $I_K \subset I_{K'}$ which proves the claim.

There are three cases for the classes $K$ disjoint from $\langle \cpgen \rangle$: 
\begin{enumerate}
\item $|I_K|=1$ for all $K$, 

\item there exists $z \in \{\kleingenone,\kleingentwo,\kleingenthree\}$ such that for every $K$ either $I_K=\{z\}$ or $I_K=\{\kleingenone,\kleingentwo,\kleingenthree\}\setminus \{z\}$,

\item $I_K=\{\kleingenone,\kleingentwo,\kleingenthree\}$ for all $K$.
\end{enumerate}
If $|I_K|=1$ for all $K$ then there must exist $i\not\equiv 0 \pmod p$ such that $\kleingenone \cpgen^i \in \superclass{\kleingenone}$. Then we have Equation (\ref{constoncp}) for $y=\kleingenone$, and hence $\{\kleingenone,\kleingenone \cpgen, \ldots, \kleingenone \cpgen^{p-1}\}$ is a subset of a superclass. Similarly we have that $\{\kleingentwo,\kleingentwo \cpgen,\ldots, \kleingentwo \cpgen^{p-1}\}$ and $\{\kleingenthree,\kleingenthree \cpgen, \ldots, \kleingenthree \cpgen^{p-1}\}$ are subsets of superclasses. Clearly they are in fact superclasses.

For the case where $I_K=\{z\}$ or $I_K=\{\kleingenone,\kleingentwo,\kleingenthree\}\setminus \{z\}$, without loss of generality assume $z=\kleingenone$. Then as above there must exist $i\not\equiv 0 \pmod p$ such that $\kleingenone \cpgen^i \in \superclass{\kleingenone}$, and again we have $\{\kleingenone,\kleingenone \cpgen, \ldots, \kleingenone \cpgen^{p-1}\}$ is a superclass. There must exist $\beta \in \{\kleingentwo,\kleingenthree\}$ and $j\not\equiv 0 \pmod p$ such that $\beta \cpgen^j \in \superclass{\kleingentwo}$. If $\beta =\kleingentwo$ then by Equation (\ref{constoncp}) $\{\kleingentwo,\kleingentwo \cpgen,\ldots, \kleingentwo \cpgen^{p-1}\}\subset \superclass{\kleingentwo}$. Therefore if $\kleingentwo \in I_K$, then $K=\superclass{\kleingentwo}$ so we see that $\{\kleingentwo,\kleingentwo \cpgen,\ldots, \kleingentwo \cpgen^{p-1},\kleingenthree,\kleingenthree \cpgen, \ldots, \kleingenthree \cpgen^{p-1}\}$ is a superclass. If $\beta =\kleingenthree$ then $\kleingenthree \cpgen^j \in \superclass{\kleingentwo}$ so by Equation (\ref{constoncp}) $\{\kleingentwo,\kleingenthree \cpgen,\ldots,\kleingenthree \cpgen^{p-1}\} \subset \superclass{\kleingentwo}$. There also exists $t\not\equiv 0 \pmod p$ such that $\kleingentwo \cpgen^t \in \superclass{\kleingenthree}$, so $\{\kleingenthree,\kleingentwo \cpgen,\ldots,\kleingentwo \cpgen^{p-1}\}\subset \superclass{\kleingenthree}$. We are then able to combine Equations (\ref{thetaeqexpanded}) and (\ref{constoncp}) twice, once with $x=\kleingentwo, y=\kleingenthree$ and once with $x=\kleingenthree, y=\kleingentwo$:
\begin{equation}
(p-1)\Msum{\theta}{\cpgenchar}(\kleingentwo)=-\Msum{\theta}{\cpgenchar}(\kleingenthree),
\end{equation}
\begin{equation}
(p-1)\Msum{\theta}{\cpgenchar}(\kleingenthree)=-\Msum{\theta}{\cpgenchar}(\kleingentwo).
\end{equation}
We then have
\begin{equation}
-(p-1)^2\Msum{\theta}{\cpgenchar}(\kleingentwo)=-\Msum{\theta}{\cpgenchar}(\kleingentwo).
\end{equation}
Therefore
\begin{equation}
\Msum{\theta}{\cpgenchar}(\kleingentwo)=\Msum{\theta}{\cpgenchar}(\kleingenthree)=0.
\end{equation}
We conclude that $\superclass{\kleingentwo}=\superclass{\kleingenthree}$, or equivalently $\{\kleingentwo,\kleingentwo\cpgen,\ldots,\kleingentwo\cpgen^{p-1},\kleingenthree,\kleingenthree\cpgen,\ldots,\kleingenthree\cpgen^{p-1}\}$ is a superclass.

Finally, we consider when $I_K=\{\kleingenone, \kleingentwo, \kleingenthree\}$ for all $K$. In this case we want to show that 
\begin{equation}
G\setminus \langle \cpgen \rangle=\{\kleingenone,\kleingenone \cpgen,\ldots,\kleingenone \cpgen^{p-1},\kleingentwo,\kleingentwo\cpgen,\ldots,\kleingentwo\cpgen^{p-1},\kleingenthree,\kleingenthree\cpgen,\ldots,\kleingenthree\cpgen^{p-1}\}
\end{equation}
is a superclass.
There exists $x\in\{\kleingenone,\kleingentwo,\kleingenthree\}$ and $i\not\equiv 0 \pmod p$ such that $x\cpgen^i \in \superclass{\kleingenone}$. If $x=\kleingenone$ then by Equation (\ref{constoncp}) $\{\kleingenone,\kleingenone \cpgen,\ldots,\kleingenone \cpgen^{p-1}\}\subset \superclass{\kleingenone}$. Since $\kleingenone \in I_K$ implies $K=\superclass{\kleingenone}$ we must have $G\setminus \langle \cpgen \rangle$ is a superclass. We now suppose that $x\not= \kleingenone$ and without loss of generality suppose that $x=\kleingentwo$. Then by Equation (\ref{constoncp}) with $y=\kleingentwo$ we have $\{\kleingenone,\kleingentwo \cpgen,\ldots,\kleingentwo \cpgen^{p-1}\}\subset \superclass{\kleingenone}$. There exists $j$ such that $\kleingenone \cpgen^j \in \superclass{\kleingentwo}$. If $j\equiv 0 \pmod p$, then $\superclass{\kleingenone}=\superclass{\kleingentwo}$. $\kleingentwo \in I_K$ implies $K=\superclass{\kleingenone}$ so $G\setminus \langle \cpgen \rangle$ is a superclass. If $j\not\equiv 0 \pmod p$, then by Equation (\ref{constoncp}) $\{\kleingentwo,\kleingenone\cpgen,\ldots,\kleingenone\cpgen^{p-1}\}\subset \superclass{\kleingentwo}$. As above, we can combine \ref{thetaeqexpanded} and \ref{constoncp} twice to conclude that $\superclass{\kleingenone}=\superclass{\kleingentwo}$. Since $\kleingenone \in I_K$ implies $K=\superclass{\kleingenone}$ we again have $G\setminus \langle \cpgen \rangle$ is a superclass.

We now see that every superclass disjoint from $\langle \cpgen \rangle$ is a union of $\langle \cpgen \rangle$-cosets. Therefore we see that $A$ is a wedge product of a supercharacter theory for $C_p$ and a supercharacter theory for $(C_p\times C_2\times C_2)/C_p$. This supercharacter theory is determined by a choice of the supercharacter theory for $C_p$ and a choice of supercharacter theory for $\irr(C_2\times C_2)$.

\section{Complementary Subgroups}

We now consider the cases where there exist complementary subgroups which are unions of superclasses.
We shall first need a few lemmas about supercharacter theories of $C_p=\langle \cpgen \rangle$. Recall that $\rootofone$ is a chosen primitive $p$th root of unity.
\begin{lemma}
\label{fromauto}
\cite[Lemma 6.9]{Hendthesis} If $p$ is prime, then every supercharacter theory of $C_p$ is generated by automorphisms.
\end{lemma}
From this we have the following lemma. We will also provide a direct proof.
\begin{lemma}
\label{samesize}
Let $A$ be a supercharacter theory of $C_p$ where $p$ is prime. Then every element of $\mathcal{K}_A$ other than $\{e\}$ has the same size, and this is equal to the size of every element of $\mathcal{X}_A$ other than $\{\operatorname{triv}\}$.
\end{lemma}
\begin{proof}
Fix a supercharacter theory $A$ of $C_p$. Let $\theta\not= \operatorname{triv}$ be a supercharacter of $A$ with $r$ irreducible characters as summands such that $r$ is minimal. For $k\not\equiv 0 \pmod p$, $\theta(\cpgen^k)$ is a sum of $r$ distinct roots of unity. Let $\chi$ be an irreducible character which is a summand of $\theta$. Because $\{\rootofone,\rootofone^2,\ldots,\rootofone^{p-1}\}$ is linearly independent over $\mathbb{Q}$ we have $\chi(\cpgen^k)=\chi(\cpgen^v)$ iff $\cpgen^k=\cpgen^v$. Hence we see that $r$ is the maximal number of elements a superclass can contain. Let $m$ be the number of superclasses different from $\{e\}$, which is equal to the number of supercharacters different from $\operatorname{triv}$. Then since $r$ is the minimal number of irreducible characters which are summands of a nontrivial supercharacter we have $rm \leq p-1$. However, since $r$ is the maximal number of elements in a superclass we see that $p-1 \leq rm$. We conclude that $rm=p-1$. Hence every superclass other than $\{e\}$ has $r$ elements, and every supercharacter other than $\operatorname{triv}$ has $r$ irreducible characters as summands.
\end{proof}
\begin{rmrk}\label{autormrk}
The size of every superclass other than $\{e\}$ is the order of the group of automorphisms which generates the supercharacter theory. Further since $\aut(C_p)\cong C_{p-1}$ for $p$ prime, and a cyclic group has at most one subgroup of a given order, we see that the size of the superclasses determines the supercharacter theory. Also the dimension determines the supercharacter theory.
\end{rmrk}

\begin{lemma}
Let $\theta\not= \operatorname{triv}$ be a supercharacter for a supercharacter theory $A$ of $C_p$ where $p$ is prime. If $\chi_1\not= \chi_2$ are summands of $\theta$ and $x\not= e$, then there exists $y\in \superclass{x}$, $y\not= x$, such that $\chi_1(x)=\chi_2(y)$.
\end{lemma}
\begin{proof}
Suppose that $\theta$ is a sum of $r$ irreducible characters. Since $x\not= e$, $\theta(x)$ is a sum of $r$ distinct roots of unity. By Lemma \ref{samesize}, $|\superclass{x}|=r$. Since $\chi_2(x)=\chi_2(z)$ iff $x=z$, we see that there are $r$ distinct values that $\chi_2$ takes when evaluating elements of $\superclass{x}$. Hence there must be a $y\in \superclass{x}$ such that $\chi_2(y)=\chi_1(x)$.
\end{proof}
\begin{lemma}
\label{disjoint}
Let $\cpgen^\ell,\cpgen^k \in C_p$ for $p$ prime, and let $\theta \not= \operatorname{triv}$ be a supercharacter for some supercharacter theory $A$ of $C_p$. If $\cpgen^\ell \notin \superclass{\cpgen^k}$ then the set of roots of unity which appear as summands of $\theta(\cpgen^\ell)$ are disjoint from those which appear as summands of $\theta(\cpgen^k)$.
\end{lemma}
\begin{proof}
If either $\cpgen^\ell=e$ or $\cpgen^k=e$, then the result clearly holds. Suppose that $\cpgen^\ell \not= e$ and $\cpgen^k \not= e$. Clearly the result also holds if $\theta$ is a single irreducible character, so suppose it is not. Suppose that $\rootofone^i$ is a summand of both $\theta(\cpgen^\ell)$ and $\theta(\cpgen^k)$. Then there exist $\chi_1$, $\chi_2 \in \irr(C_p)$, $\chi_1\not=\chi_2$ such that $\chi_1,\chi_2$ are summands of $\theta$ and $\chi_1(\cpgen^\ell)=\chi_2(\cpgen^k)=\rootofone^i$. However, by the previous lemma there exists a $\cpgen^{\ell'}\in \superclass{\cpgen^k}$ such that $\chi_1(\cpgen^{\ell'})=\rootofone^i$, which is a contradiction.
\end{proof}

Suppose that $A$ is a supercharacter theory such that $\widehat{H_1},\widehat{H_2}\in A$ for a complementary pair $H_1,H_2$ of $C_p\times C_2\times C_2$. This may occur in two ways: $H_1\cong C_p \times C_2$ and $H_2\cong C_2$ or $H_1 \cong C_p$ and $H_2 \cong C_2 \times C_2$. We leave the proof of the following to the reader:
\begin{lemma}\label{refine}
 If $H_1$ and $H_2$ are a complementary pair and $\widehat{H_1},\widehat{H_2}\in A$, then $A$ is a refinement of the direct product theory of $A|_{H_1}$ and $A|_{H_2}$. 
\end{lemma}
If $H_1\cong C_p \times C_2$ and $H_2\cong C_2$, then we see that $A$ must be the direct product supercharacter theory as there is no supercharacter theory $A'$ which is a refinement of $A$ such that $A'|_{H_2}=A|_{H_2}$. Since $C_p \times C_2$ is cyclic, all its supercharacter theories have already been classified, see \cite{cyclicSchur} and \cite{cyclicSchurII}, and Lemma \ref{CpC2}. Hence the case we must consider is $H_1 \cong C_p$ and $H_2 \cong C_2 \times C_2$. Similarly if $A|_{H_1}$ or $A|_{H_2}$ is the minimal supercharacter theory, then $A$ must be the direct product supercharacter theory. We shall divide the remaining possibilities into two cases by considering the dimension of $A|_{\langle \kleingenone,\kleingentwo \rangle}$. By definition the dimension can't be one, and if the dimension is four $A|_{\langle \kleingenone,\kleingentwo \rangle}$ is the minimal supercharacter theory. Case 3 will be when the dimension is equal to three, and Case 4 will be when the dimension is equal to two.

\subsection{Case 3}

Suppose that $H_1=\langle \cpgen \rangle$, $H_2 =\{e,\kleingenone,\kleingentwo,\kleingenthree\}$ and $\widehat{H_1},\widehat{H_2}\in A$. By Lemma \ref{dual}, this implies that $\kleingenonechar+\kleingentwochar+\kleingenthreechar$ and $\widehat{\langle \cpgenchar \rangle}$ are sums of supercharacters. We consider the case where $A|_{\langle \cpgen \rangle}$ is not the minimal supercharacter theory for $\langle \cpgen \rangle$, and the dimension of $A|_{\langle \kleingenone,\kleingentwo\rangle}$ is equal to three. Without loss of generality we let $\{\kleingenone\}$ and $\{\kleingentwo,\kleingenthree\}$ be superclasses, which implies by Lemma \ref{dual} that $\kleingenonechar$ and $\kleingentwochar+\kleingenthreechar$ are supercharacters. Note that this implies $\widehat{\langle \cpgenchar,\kleingenonechar \rangle}$ is a sum of supercharacters. Let $T_{\mathcal{K}}$ be the set of superclasses which are disjoint from $\langle \cpgen,\kleingenone \rangle$ and $\langle \kleingenone,\kleingentwo \rangle$ and let $K\in T_{\mathcal{K}}$. Similarly let $T_{\mathcal{X}}$ be the set of supercharacters which are not a summand of $\widehat{\irr(\langle \cpgenchar,\kleingenonechar \rangle)}$ or $\widehat{\irr(\langle \kleingenonechar, \kleingentwochar \rangle)}$, and let $\theta \in T_{\mathcal{X}}$. We note that $|T_{\mathcal{K}}|=|T_{\mathcal{X}}|$. $\{\Mset{\theta}{\kleingentwochar}\}_{\theta \in T_{\mathcal{X}}}$ is a partition of $\irr(\langle \cpgen \rangle)\setminus \{ \operatorname{triv}\}$. Since $\Mset{\theta}{\kleingentwochar}=\Mset{\kleingenonechar\theta}{\kleingenthreechar}$ and $\theta \in T_{\mathcal{X}}$ iff $\kleingenonechar\theta \in T_{\mathcal{X}}$ we see that $\{\Mset{\theta}{\kleingenthreechar}\}_{\theta \in T_{\mathcal{X}}}$ is the same partition. Similarly $\{ \Mset{K}{\kleingentwo}\}_{K \in T_{\mathcal{K}}} = \{ \Mset{K}{\kleingenthree}\}_{K \in T_{\mathcal{K}}}$ are partitions of $\langle \cpgen \rangle\setminus \{e\}$. We want to show that these partitions define a supercharacter theory of $C_p$ which is a refinement of $A|_{\langle \cpgen \rangle}$. 

Since $\theta(\kleingentwo)=\theta(\kleingenthree)$ we have
\begin{equation}
|\Mset{\theta}{\kleingentwochar}|=|\Mset{\theta}{\kleingenthreechar}|.
\end{equation}
If $\Mset{\theta}{\kleingentwochar}\cap \Mset{\theta}{\kleingenthreechar}\not= \emptyset$, then $\theta=\kleingenonechar\theta$ so $\Mset{\theta}{\kleingentwochar}=\Mset{\theta}{\kleingenthreechar}$. Hence $\Mset{\theta}{\kleingentwochar}$ and $\Mset{\theta}{\kleingenthreechar}$ are either disjoint or equal.

Suppose they are disjoint. By Lemma \ref{refine}, $\theta$ is a summand of $\beta (\kleingentwochar+\kleingenthreechar)$ for some supercharacter $\beta$ a summand of $\widehat{\irr(\langle\cpgen\rangle)}$.
\begin{equation}
\theta(\kleingentwochar+\kleingenthreechar)=(\Msum{\theta}{\kleingentwochar}+\Msum{\theta}{\kleingenthreechar})+(\Msum{\theta}{\kleingentwochar}+\Msum{\theta}{\kleingenthreechar})\kleingenonechar.
\end{equation}
$\Msum{\theta}{\kleingentwochar}+ \Msum{\theta}{\kleingenthreechar}$ is a summand of $\beta$, so $\Msum{\theta}{\kleingentwochar}+ \Msum{\theta}{\kleingenthreechar}=\beta$.
If $\cpgen^{\ell}\kleingentwo \in \superclass{\cpgen^k \kleingentwo}$ for $\ell\not\equiv k \pmod p$ then
\begin{equation}
\theta(\cpgen^\ell \kleingentwo)=\theta(\cpgen^k \kleingentwo)
\end{equation}
\begin{equation}
\Msum{\theta}{\kleingentwochar}(\cpgen^{\ell})-\Msum{\theta}{\kleingenthreechar}(\cpgen^{\ell})=\Msum{\theta}{\kleingentwochar}(\cpgen^k)-\Msum{\theta}{\kleingenthreechar}(\cpgen^k).
\end{equation}
Since $\Mset{\theta}{\kleingentwochar}\cap \Mset{\theta}{\kleingenthreechar}=\emptyset$ this implies
\begin{equation}
\Msum{\theta}{\kleingentwochar}(\cpgen^{\ell})=\Msum{\theta}{\kleingentwochar}(\cpgen^k),
\end{equation}
\begin{equation}
\Msum{\theta}{\kleingenthreechar}(\cpgen^{\ell})=\Msum{\theta}{\kleingenthreechar}(\cpgen^k).
\end{equation}
Similarly if $\cpgen^{\ell}\kleingenthree \in \superclass{\cpgen^k \kleingenthree}$ for $\ell\not\equiv k \pmod p$ the above equations hold.

Since $|T_{\mathcal{X}}|=|T_{\mathcal{K}}|$, we see that 
\begin{equation}
(\{\Mset{\theta}{\kleingenthreechar}\}_{\theta \in T_{\mathcal{X}}},\{ \Mset{K}{\kleingenthree}\}_{K \in T_{\mathcal{K}}})=(\{\Mset{\theta}{\kleingentwochar}\}_{\theta \in T_{\mathcal{X}}},\{ \Mset{K}{\kleingentwo}\}_{K \in T_{\mathcal{K}}})
\end{equation}
is a supercharacter theory of $C_p$. It is clearly a refinement of $A|_{\langle \cpgen \rangle}$ since $A$ is a refinement of a direct product theory by Lemma \ref{refine}. We further observe that for $\cpgen^\ell\kleingentwo \in \superclass{\cpgen^k\kleingenthree}$
\begin{equation}
\theta(\cpgen^\ell\kleingentwo)=\theta(\cpgen^k\kleingenthree)
\end{equation}
\begin{equation}
\Msum{\theta}{\kleingentwochar}(\cpgen^\ell)-\Msum{\theta}{\kleingenthreechar}(\cpgen^\ell)=-\Msum{\theta}{\kleingentwochar}(\cpgen^k)+\Msum{\theta}{\kleingenthreechar}(\cpgen^k)
\end{equation}
\begin{equation}
\Msum{\theta}{\kleingentwochar}(\cpgen^\ell)=\Msum{\theta}{\kleingenthreechar}(\cpgen^k).
\end{equation}

Similarly, for $K\in T_{\mathcal{K}}$ we have that $\Mset{K}{\kleingentwo}$ and $\Mset{K}{\kleingenthree}$ are either equal or disjoint. Suppose that there exists $K$ such that $\Mset{K}{\kleingentwo}= \Mset{K}{\kleingenthree}$. Let $\cpgen^k \kleingentwo$, $\cpgen^k \kleingenthree \in K$ and $\theta \in T_{\mathcal{X}}$. Then 
\begin{equation}
\theta(\cpgen^k \kleingentwo)=\theta(\cpgen^k \kleingenthree)
\end{equation}
\begin{equation}
\theta(\cpgen^k \kleingentwo)=\Msum{\theta}{\kleingentwochar}(\cpgen^k)-\Msum{\theta}{\kleingenthreechar}(\cpgen^k)=-\theta(\cpgen^k \kleingenthree)
\end{equation}
\begin{equation}
\theta(\cpgen^k \kleingentwo)=\theta(\cpgen^k \kleingenthree)=0.
\end{equation}
So $\Msum{\theta}{\kleingentwochar}$ and $\Msum{\theta}{\kleingenthreechar}$ have to take the same value on $\Mset{K}{\kleingentwo}=\Mset{K}{\kleingenthree}$. For a supercharacter theory of $C_p$ where $p$ is prime, it is impossible for distinct supercharacters to take the same value on a superclass. Therefore either $\Mset{\theta}{\kleingentwochar}=\Mset{\theta}{\kleingenthreechar}$ and hence $A$ is a direct product theory, or $\Mset{K}{\kleingentwo}\not= \Mset{K}{\kleingenthree}$ for every $K\in T_{\mathcal{K}}$. Since $|T_{\mathcal{K}}|=|T_{\mathcal{X}}|$, the latter case implies $\Mset{\theta}{\kleingentwochar}\not= \Mset{\theta}{\kleingenthreechar}$ for every supercharacter $\theta \in T_{\mathcal{X}}$.

This concludes our consideration of the conditions that every supercharacter theory of the form of Case 3 must satisfy. We now consider the sufficient direction and show that any partition that satisfies these conditions is indeed a supercharacter theory. Given a set $X$ we will denote the set of all subsets of $X$ by $\mathcal{P}(X)$. We will consider the potential supercharacter theory $(\mathcal{X},\mathcal{K})$ of $G=\langle \cpgen,\kleingenone,\kleingentwo\rangle$ where $\mathcal{X}\subset \mathcal{P}(\irr(G))$, $\mathcal{K} \subset \mathcal{P}(G)$, $|\mathcal{X}|=|\mathcal{K}|$ and further the following are satisfied:
\begin{itemize}
\item
$(\mathcal{X}\cap \mathcal{P}(\irr(\langle \cpgen \rangle)),\mathcal{K}\cap\mathcal{P}(\langle \cpgen \rangle))$ is a non-minimal supercharacter theory for $C_p$ and we will denote the corresponding algebra by $B$.
\item
$\{\kleingenone\}, \{\kleingentwo,\kleingenthree\}\in \mathcal{K}$ and $\{\kleingenonechar\}, \{\kleingentwochar,\kleingenthreechar\}\in \mathcal{X}$.
\item
$(\mathcal{X}\cap \mathcal{P}(\irr(\langle \cpgen, \kleingenone \rangle)),\mathcal{K}\cap\mathcal{P}(\langle \cpgen, \kleingenone \rangle))$ is the direct product supercharacter theory of $B$ and the unique supercharacter theory of $\langle \kleingenone \rangle$.
\end{itemize}
If $(\mathcal{X},\mathcal{K})$ is a direct product supercharacter theory we are done, so assume that it is not. Suppose that 
\begin{equation}
\{\Mset{K}{\kleingentwo}|K\in \mathcal{K},K\subset \langle\cpgen\rangle \times \{\kleingentwo,\kleingenthree\} \}=\{\Mset{K}{\kleingenthree}|K\in \mathcal{K},K\subset \langle\cpgen\rangle \times \{\kleingentwo,\kleingenthree\}\}
\end{equation}
 is a supercharacter theory $B'$ of $C_p$ which is a refinement of $B$ satisfying the following condition: for every superclass $U\not= \{e\}$ of $B'$ there is a superclass $V$ of $B'$ satisfying 
\begin{enumerate}
\item[(A1)]
$|U|=|V|$ 
\item[(A2)]
$U\cap V=\emptyset$
\item[(A3)]
$U\cup V$ is a superclass of $B$
\item[(A4)]
$(U\times \{\kleingentwo\})\cup (V\times \{\kleingenthree\}) \in \mathcal{K}$
\item[(A5)]
$(U\times \{\kleingenthree\})\cup (V\times \{\kleingentwo\}) \in \mathcal{K}$.
\end{enumerate}
Also, the analogous conditions hold for the supercharacters as well, with $\sigma, \sigma'$ in the roles of $U,V$. We also assume that for every such supercharacter of $B'$ $\sigma\not= \operatorname{triv}$ we have for all $x\in U$, $y \in V$
\begin{equation}\label{charsplit}
\sigma(x)=\sigma'(y).
\end{equation}
This is well defined since $\sigma+\sigma'$ is a supercharacter for $B$, and hence is constant on $U\cup V$. Therefore Equation (\ref{charsplit}) is equivalent to $\sigma(y)=\sigma'(x)$. We can view these assumptions as stating that every superclass of $B'$ is a superclass of $B$ split into two halves which satisfy Equation (\ref{charsplit}). Note that $U$ and $V$ correspond to $\Mset{K}{\kleingentwo}$ and $\Mset{K}{\kleingenthree}$ and similarly $\sigma$ and $\sigma'$ correspond to $\Msum{\theta}{\kleingentwochar}$ and $\Msum{\theta}{\kleingenthreechar}$. This concludes our listing of assumptions, we shall now show that $(\mathcal{X},\mathcal{K})$ is a supercharacter theory for $G$. 

We observe that $|\mathcal{X}|=|\mathcal{K}|$, hence we only need to show that all potential supercharacters are constant on all potential superclasses.
We begin with $K\in T_{\mathcal{K}}$ and $\theta \in T_{\mathcal{X}}$. $\Msum{\theta}{\kleingentwochar}$ and $\Msum{\theta}{\kleingenthreechar}$ are constant on $\Mset{K}{\kleingentwo}$ and $\Mset{K}{\kleingenthree}$ because they are supercharacters and superclasses for $B'$. By Equation (\ref{charsplit}) we have that $\theta$ is constant on all of $K$. Since a supercharacter of $B$ is constant on $\Mset{K}{\kleingentwo}\cup \Mset{K}{\kleingenthree}$ by assumption (A3), it is constant on $\Mset{K}{\kleingentwo}$ and $\Mset{K}{\kleingenthree}$, so it is constant on $K$. Similarly for the supercharacters of $\irr(\langle \cpgen \rangle )\times\{ \kleingenonechar\}$. Clearly $\kleingenonechar$ and $\kleingentwochar+\kleingenthreechar$ are constant on $K$. Hence every potential supercharacter from $\mathcal{X}$ is constant on $K\in T_{\mathcal{K}}$.

Now let $K\in \mathcal{K}$ be a superclass which is a subset of $\langle \cpgen \rangle$. If $\cpgen^\ell \in K$, then for $\theta \in T_{\mathcal{X}}$ we have $\theta(\cpgen^\ell)=(\Msum{\theta}{\kleingentwochar}+\Msum{\theta}{\kleingenthreechar})(\cpgen^\ell)$. Since $\Msum{\theta}{\kleingentwochar}+\Msum{\theta}{\kleingenthreechar}$ is a supercharacter for $B$, we see that $\theta$ is constant on $K$. The supercharacters of $\irr(\langle \cpgen \rangle)$ and $\irr(\langle \cpgen \rangle)\times \{\kleingenonechar\}$ are clearly constant on $K$. $\kleingentwochar+\kleingenthreechar$ takes the constant value 2 and $\kleingenonechar$ takes the constant value $1$ on $K$. In a similar fashion, we see that all potential supercharacters in $\mathcal{X}$ are constant on the superclasses which are contained in $\langle \cpgen \rangle \times \{\kleingenone\}$.

Finally we consider the superclass $\{\kleingentwo,\kleingenthree\}$. For $\theta \in T_{\mathcal{X}}$ we have 
\begin{align}
&\theta(\kleingentwo)=|\Mset{\theta}{\kleingentwochar}|-|\Mset{\theta}{\kleingenthreechar}|=0,
\\
&\theta(\kleingenthree)=-|\Mset{\theta}{\kleingentwochar}|+|\Mset{\theta}{\kleingenthreechar}|=0.
\end{align}
Clearly any supercharacter of $\irr(\langle \cpgen \rangle)$ or $\irr(\langle \cpgen \rangle)\times \{\kleingenonechar\}$ is constant on $\{\kleingentwo,\kleingenthree\}$. $\kleingenonechar$ takes the value $-1$, and $\kleingentwochar+\kleingenthreechar$ takes the value 0 on $\{\kleingentwo,\kleingenthree\}$.
Therefore every potential supercharacter is constant on every potential superclass. We conclude that $(\mathcal{X},\mathcal{K})$ defines a supercharacter theory of $G$.

The above argument describes the supercharacter theories of Case 3. We will now present an alternate proof of the sufficient direction above, as it is useful to see that every supercharacter theory of Case 3 which is not a direct product can be generated by a group of automorphisms.  Let $A$ be such a supercharacter theory, and let $B$ be the supercharacter theory $A|_{\langle \cpgen \rangle}$. Let $B'$ be a refinement of $B$ satisfying $(A1)$ through $(A5)$. By Lemma \ref{fromauto} every supercharacter theory of $C_p$ can be generated by a subgroup of $\aut(C_p)\cong C_{p-1}$. Since every subgroup of a cyclic group is cyclic, there exists $\phi_0 \in \aut(\langle \cpgen \rangle)$ such that$\langle \phi_0 \rangle$ generates $B$. Similarly, there exists $\phi'$ where $\langle \phi' \rangle$ generates the supercharacter theory $B'$. By Lemma \ref{samesize} we know that every superclass other than $\{e\}$ in $B'$ is the same size, and by (A3) the size of every superclass other then $\{e\}$ in $B$ is double that size. Further, it is clear that this size equals $|\langle \phi' \rangle |$. Since a cyclic group has at most one subgroup of a given order we have $\langle \phi' \rangle=\langle \phi_0^2 \rangle$. We now let $\phi \in \aut(G)$ be defined by $\phi(\cpgen)=\phi_0(\cpgen)$, $\phi(\kleingenone)=\kleingenone$, $\phi(\kleingentwo)=\kleingentwo$, and $\phi(\kleingenthree)=\kleingenthree$.

Let $\psi \in \aut(G)$ be such that $\psi( \cpgen)=\cpgen$, $\psi(\kleingenone)=\kleingenone$, $\psi(\kleingentwo)=\kleingenthree$, and $\psi(\kleingenthree)=\kleingentwo$. Let $A'$ be the supercharacter theory generated by the group of automorphisms $\langle \phi \circ \psi \rangle$. We will show that $A=A'$. We see that $\{\kleingenone\}$ and $\{\kleingentwo, \kleingenthree \}$ are superclasses for $A'$, and $A|_{\langle \cpgen \rangle}=A'|_{\langle \cpgen \rangle}$. Let $\ell\not \equiv 0 \pmod p$ and let $K_0=\superclass{\cpgen^\ell \kleingentwo} \in \mathcal{K}(A')$. Then $\Mset{K_0}{\kleingentwo}=\{\cpgen^\ell,\phi^2(\cpgen^\ell),\phi^4(\cpgen^\ell),\ldots\}$ and $\Mset{K_0}{\kleingenthree}=\{\phi(\cpgen^\ell),\phi^3(\cpgen^\ell),\phi^5(\cpgen^\ell),\ldots\}$. Hence we see that $\{\Mset{K}{\kleingentwo}\}_{K\in \mathcal{K}(A')}=\{\Mset{K}{\kleingentwo}\}_{K\in \mathcal{K}(A)}$, $\{\Mset{K}{\kleingenthree}\}_{K\in \mathcal{K}(A')}=\{\Mset{K}{\kleingenthree}\}_{K\in \mathcal{K}(A)}$, and $\Mset{K_0}{b}\cup \Mset{K_0}{c}\in \mathcal{K}(A|_{\langle \cpgen \rangle})$. Hence $A$ and $A'$ have the same superclasses, so we conclude that $A=A'$.

For a supercharacter theory generated by automorphisms as above, see Example \ref{exwithM}.

\subsection{Case 4}

We now consider the final case for $p$ odd. Again suppose that $H_1=\langle \cpgen \rangle$ and $H_2 =\{e,\kleingenone,\kleingentwo,\kleingenthree\}$, and $\widehat{H_1},\widehat{H_2}\in A$. Suppose that $A|_{\langle \cpgen \rangle}$ is not the minimal supercharacter theory, and that the dimension of $A|_{\langle \kleingenone,\kleingentwo \rangle}$ is equal to 2 which is equivalent to $\{\kleingenone,\kleingentwo,\kleingenthree\}$ is a superclass. Let $T_{\mathcal{K}}$ be the set of superclasses of $A$ which are disjoint from $H_1$ and $H_2$, and let $T_{\mathcal{X}}$ be the set of supercharacters which are not a summand of $\widehat{\irr(H_1)}$ or $\widehat{\irr(H_2)}$. Again $|T_{\mathcal{K}}|=|T_{\mathcal{X}}|$. Let $K \in T_{\mathcal{K}}$ and $\theta \in T_{\mathcal{X}}$. Since $\{\kleingenone,\kleingentwo,\kleingenthree\}$ is a superclass, there is a constant $r$ such that
\begin{equation}
((\cpgen+\cpgen^2+\ldots+\cpgen^{p-1})\widehat{K})\hadaprod (\kleingenone+\kleingentwo+\kleingenthree)=r(\kleingenone+\kleingentwo+\kleingenthree) \in A.
\end{equation}
Hence $|\Mset{K}{\kleingenone}|=|\Mset{K}{\kleingentwo}|=|M(K,\kleingenthree)|=r$. Similarly $|\Mset{\theta}{\kleingenonechar}|=|\Mset{\theta}{\kleingentwochar}|=|\Mset{\theta}{\kleingenthreechar}|$. Then $(\kleingenone+\kleingentwo+\kleingenthree)\widehat{K}$ equals:
\begin{align}\notag
&(\Msum{K}{\kleingenone}+\Msum{K}{\kleingentwo}+\Msum{K}{\kleingenthree})+(\Msum{K}{\kleingenthree}\kleingenone+\Msum{K}{\kleingenone}\kleingentwo+\Msum{K}{\kleingentwo}\kleingenthree)+
\\ 
&(\Msum{K}{\kleingentwo}\kleingenone+\Msum{K}{\kleingenthree}\kleingentwo + \Msum{K}{\kleingenone}\kleingenthree).
\end{align}
By Lemma \ref{refine} $K$ is a subset of $\alpha \times \{\kleingenone,\kleingentwo,\kleingenthree\}$ for some superclass $\alpha$ of $A|_{\langle \cpgen \rangle}$, so we have for some constant $t$, $\Msum{K}{\kleingenone}+ \Msum{K}{\kleingentwo} + \Msum{K}{\kleingenthree}=t\widehat{\alpha}$. Then the remaining terms
\begin{equation}
W=(\Msum{K}{\kleingenthree}+\Msum{K}{\kleingentwo})\kleingenone+(\Msum{K}{\kleingenone}+\Msum{K}{\kleingenthree})\kleingentwo+(\Msum{K}{\kleingentwo}+\Msum{K}{\kleingenone})\kleingenthree \in A.
\end{equation}
If $\widehat{K}$ is not a summand of $W$ then $\Mset{K}{\kleingenone}$, $\Mset{K}{\kleingentwo}$, and $\Mset{K}{\kleingenthree}$ are pairwise disjoint. Suppose $\widehat{K}$ is a summand. Without loss of generality, suppose $\cpgen^{\ell} \in \Mset{K}{\kleingenone}\cap \Mset{K}{\kleingentwo}$ with $\ell\not\equiv 0 \pmod p$. Then
\begin{equation}
\theta(\cpgen^{\ell}\kleingenone)=\theta(\cpgen^{\ell}\kleingentwo)
\end{equation}
and further by evaluating we see that
\begin{align}\notag
&\Msum{\theta}{\kleingenonechar}(\cpgen^{\ell})-\Msum{\theta}{\kleingentwochar}(\cpgen^{\ell})-\Msum{\theta}{\kleingenthreechar}(\cpgen^{\ell})=
\\
&-\Msum{\theta}{\kleingenonechar}(\cpgen^{\ell})+\Msum{\theta}{\kleingentwochar}(\cpgen^{\ell})-\Msum{\theta}{\kleingenthreechar}(\cpgen^{\ell})
\end{align}
which becomes
\begin{equation}
\Msum{\theta}{\kleingenonechar}(\cpgen^{\ell})-\Msum{\theta}{\kleingentwochar}(\cpgen^{\ell})=-\Msum{\theta}{\kleingenonechar}(\cpgen^{\ell})+\Msum{\theta}{\kleingentwochar}(\cpgen^{\ell})
\end{equation}
\begin{equation}
\Msum{\theta}{\kleingenonechar}(\cpgen^{\ell})=\Msum{\theta}{\kleingentwochar}(\cpgen^{\ell}).
\end{equation}
That implies $\Mset{\theta}{\kleingenonechar}=\Mset{\theta}{\kleingentwochar}$. Since $\Msum{\theta}{\kleingenonechar}+\Msum{\theta}{\kleingentwochar}+\Msum{\theta}{\kleingenthreechar}$ is a multiple of a $C_p$ supercharacter, we have $\Mset{\theta}{\kleingenthreechar}=\Mset{\theta}{\kleingentwochar}$ because $\Msum{\theta}{\kleingenonechar}+\Msum{\theta}{\kleingentwochar}+\Msum{\theta}{\kleingenthreechar}=2\Msum{\theta}{\kleingentwochar}+\Msum{\theta}{\kleingenthreechar}$ implies that some terms have multiplicity at least 2, hence all terms must have multiplicity 3. Therefore every supercharacter $\theta \in T_{\mathcal{X}}$ satisfies $\theta=\Msum{\theta}{\kleingenonechar}(\kleingenonechar+\kleingentwochar+\kleingenthreechar)$ hence $A$ is a direct product supercharacter theory. Therefore for all $K\in T_{\mathcal{K}}$ we have $\Mset{K}{\kleingenone}=\Mset{K}{\kleingentwo}=\Mset{K}{\kleingenthree}$.

Hence either $\Mset{K}{\kleingenone}$, $\Mset{K}{\kleingentwo}$, and $\Mset{K}{\kleingenthree}$ are all equal or all pairwise disjoint. Further, if there is one superclass $K\in T_{\mathcal{K}}$ with $\Mset{K}{\kleingenone}= \Mset{K}{\kleingentwo}=\Mset{K}{\kleingenthree}$, then it must be true for all superclasses in $T_{\mathcal{K}}$. Also if $\Mset{K}{\kleingenone}=\Mset{K}{\kleingentwo}=\Mset{K}{\kleingenthree}$ then $W=2\widehat{K}$ since they both have $6r$ terms. It is clear that if $\Mset{K}{\kleingenone}$, $\Mset{K}{\kleingentwo}$, and $\Mset{K}{\kleingenthree}$ are pairwise disjoint for all $K\in T_{\mathcal{K}}$, then $\Mset{\theta}{\kleingenonechar}$, $\Mset{\theta}{\kleingentwochar}$, and $\Mset{\theta}{\kleingenthreechar}$ are pairwise disjoint for all $\theta\in T_{\mathcal{X}}$.

Suppose that $\Mset{K}{\kleingenone}$, $\Mset{K}{\kleingentwo}$, and $\Mset{K}{\kleingenthree}$ are pairwise disjoint for all $K\in T_{\mathcal{K}}$. Fix a $K\in T_{\mathcal{K}}$ and let $\cpgen^s, \cpgen^t \in \Mset{K}{\kleingenone}$. Then
\begin{equation}
\theta(\cpgen^s \kleingenone)=\theta(\cpgen^t \kleingenone).
\end{equation}
Evaluating gives
\begin{align}\notag
&\Msum{\theta}{\kleingenonechar}(\cpgen^s)-\Msum{\theta}{\kleingentwochar}(\cpgen^s)-\Msum{\theta}{\kleingenthreechar}(\cpgen^s)=
\\
&\Msum{\theta}{\kleingenonechar}(\cpgen^t)-\Msum{\theta}{\kleingentwochar}(\cpgen^t)-\Msum{\theta}{\kleingenthreechar}(\cpgen^t).
\end{align}
Since $\Mset{\theta}{\kleingenonechar}$, $\Mset{\theta}{\kleingentwochar}$, and $\Mset{\theta}{\kleingenthreechar}$ are pairwise disjoint we have
\begin{equation}
\Msum{\theta}{\kleingenonechar}(\cpgen^s)=\Msum{\theta}{\kleingenonechar}(\cpgen^t).
\end{equation}
Therefore $\Msum{\theta}{\kleingenonechar}$ is constant on $\Mset{K}{\kleingenone}$. Similarly $\Msum{\theta}{\kleingentwochar}$ is constant on $\Mset{K}{\kleingentwo}$ and $\Msum{\theta}{\kleingenthreechar}$ is constant on $\Mset{K}{\kleingenthree}$. We conclude that the three pairs
\begin{equation}
\left(\{\Mset{\theta}{\kleingenonechar}\}_{\theta \in T_{\mathcal{X}}},\{\Mset{K}{\kleingenone}\}_{K \in T_{\mathcal{K}}}\right),
\end{equation}
\begin{equation}
\left(\{\Mset{\theta}{\kleingentwochar}\}_{\theta \in T_{\mathcal{X}}},\{\Mset{K}{\kleingentwo}\}_{K \in T_{\mathcal{K}}}\right),
\end{equation}
\begin{equation}
\left(\{\Mset{\theta}{\kleingenthreechar}\}_{\theta \in T_{\mathcal{X}}},\{\Mset{K}{\kleingenthree}\}_{K \in T_{\mathcal{K}}}\right),
\end{equation}
are all supercharacter theories for $C_p$ after adjoining $\{\operatorname{triv}\}$ and $\{e\}$. Using Lemma \ref{fromauto} and recalling Remark \ref{autormrk}, we see that since they are all the same dimension and $p$ is prime, they are all the same supercharacter theory.

Suppose $\cpgen^{k}\in \Mset{K}{\kleingenone}$ and $\cpgen^{\ell} \in \Mset{K}{\kleingentwo}$, then
\begin{equation}
\theta(\cpgen^k\kleingenone)=\theta(\cpgen^{\ell} \kleingentwo).
\end{equation}
Evaluating gives
\begin{align}\notag
&\Msum{\theta}{\kleingenonechar}(\cpgen^k)-\Msum{\theta}{\kleingentwochar}(\cpgen^k)-\Msum{\theta}{\kleingenthreechar}(\cpgen^k)=
\\ \label{case4maineq}
&-\Msum{\theta}{\kleingenonechar}(\cpgen^{\ell})+\Msum{\theta}{\kleingentwochar}(\cpgen^{\ell})-\Msum{\theta}{\kleingenthreechar}(\cpgen^{\ell}).
\end{align}
For any $j$ and $\kleingenxchar$ and $\kleingenychar$ where $x,y\in \{a,b,c\}$ and $x\not=y$ we observe that the set of roots of unity which appear as summands of $\Msum{\theta}{\kleingenxchar}(\cpgen^j)$ and $\Msum{\theta}{\kleingenychar}(\cpgen^j)$ are disjoint. Therefore 
\begin{equation}
\Msum{\theta}{\kleingenonechar}(\cpgen^k)=\Msum{\theta}{\kleingentwochar}(\cpgen^{\ell}).
\end{equation}
Hence for a fixed superclass $K$, $\cpgen^k\in \Mset{K}{\kleingenone}$, $\cpgen^{\ell} \in \Mset{K}{\kleingentwo}$, and $\cpgen^r \in \Mset{K}{\kleingenthree}$, we have by symmetry
\begin{equation}
\Msum{\theta}{\kleingenonechar}(\cpgen^k)=\Msum{\theta}{\kleingentwochar}(\cpgen^{\ell})=\Msum{\theta}{\kleingenthreechar}(\cpgen^r).
\end{equation}
By Lemma \ref{disjoint}, we have for $\cpgen^k \in \Mset{K}{\kleingenone}$ and $\cpgen^{\ell} \in \Mset{K}{\kleingentwo}$ that the set of roots of unity which appear as summands of $\Msum{\theta}{\kleingenthreechar}(\cpgen^k)$ is disjoint from those in $\Msum{\theta}{\kleingenthreechar}(\cpgen^{\ell})$.
Returning to Equation (\ref{case4maineq}), this implies
\begin{equation}
\Msum{\theta}{\kleingentwochar}(\cpgen^k)=\Msum{\theta}{\kleingenthreechar}(\cpgen^{\ell}),
\end{equation}
\begin{equation}
\Msum{\theta}{\kleingenonechar}(\cpgen^{\ell})=\Msum{\theta}{\kleingenthreechar}(\cpgen^k).
\end{equation}
Then we proceed in a similar fashion to get

\begin{equation}
\Msum{\theta}{\kleingenonechar}(\cpgen^r)=\Msum{\theta}{\kleingentwochar}(\cpgen^k),
\end{equation}
\begin{equation}
\Msum{\theta}{\kleingenthreechar}(\cpgen^k)=\Msum{\theta}{\kleingentwochar}(\cpgen^r).
\end{equation}
Note that the final set of equations we get by symmetry is redundant, as we may already conclude
\begin{equation}
\Msum{\theta}{\kleingentwochar}(\cpgen^k)=\Msum{\theta}{\kleingenthreechar}(\cpgen^{\ell})=\Msum{\theta}{\kleingenonechar}(\cpgen^r),
\end{equation}
\begin{equation}
\Msum{\theta}{\kleingenthreechar}(\cpgen^k)=\Msum{\theta}{\kleingenonechar}(\cpgen^{\ell})=\Msum{\theta}{\kleingentwochar}(\cpgen^r).
\end{equation}

Suppose that the supercharacter theories $A|_{\langle \cpgen \rangle}$ and $\left( \{\Mset{K}{\kleingenone}\}_{K \in T_{\mathcal{K}}},\{\Mset{\theta}{\kleingenonechar}\}_{\theta \in T_{\mathcal{X}}}\right)$ are given. Then given a superclass $\alpha \in A|_{\langle \cpgen \rangle}$, we have $\alpha=\alpha_1\cup \alpha_2 \cup \alpha_3$ where $\alpha_1,\alpha_2$, and $\alpha_3$ are superclasses in the supercharacter theory $\left( \{\Mset{K}{\kleingenone}\}_{K\in T_{\mathcal{K}}},\{\Mset{\theta}{\kleingenonechar}\}_{\theta\in T_{\mathcal{X}}}\right)$. Then there are two possibilities: either we have for $A$ the three superclasses
\begin{equation}\label{Achoice1}
\{\alpha_1\kleingenone \cup \alpha_2\kleingentwo \cup \alpha_3\kleingenthree\}, \{\alpha_2\kleingenone \cup \alpha_3\kleingentwo \cup \alpha_1\kleingenthree\}, \{\alpha_3\kleingenone \cup \alpha_1\kleingentwo \cup \alpha_2\kleingenthree\}
\end{equation} 
or
\begin{equation}\label{Achoice2}
\{\alpha_1\kleingenone \cup \alpha_3\kleingentwo \cup \alpha_2\kleingenthree\}, \{\alpha_2\kleingenone \cup \alpha_1\kleingentwo \cup \alpha_3\kleingenthree\}, \{\alpha_3\kleingenone \cup \alpha_2\kleingentwo \cup \alpha_1\kleingenthree\}.
\end{equation} 
The situation is similar for $\theta \in T_{\mathcal{X}}$. However, we do not have independent possibilities for each $K$. The above equations tell us that given a single fixed superclass $K \in T_{\mathcal{K}}$, for any $\theta \in T_{\mathcal{X}}$ the value of $\Msum{\theta}{\kleingenonechar}$ on $\Mset{K}{\kleingenone}$, $\Mset{K}{\kleingentwo}$, and $\Mset{K}{\kleingenthree}$ determines the values of $\Msum{\theta}{\kleingentwochar}$ and $\Msum{\theta}{\kleingenthreechar}$. Hence the value of a single $K$ as either Equation (\ref{Achoice1}) or (\ref{Achoice2}) determines all supercharacters in $T_{\mathcal{X}}$ and hence the value as Equation (\ref{Achoice1}) or (\ref{Achoice2}) for every other $K\in T_{\mathcal{K}}$ is determined.

To understand the two possibilities, consider the algebra automorphism $\phi:\mathbb{C}G\rightarrow \mathbb{C}G$ defined by $\phi(\cpgen)=\cpgen$, $\phi(\kleingenone)=\kleingenone$, $\phi(\kleingentwo)=\kleingenthree$, and $\phi(\kleingenthree)=\kleingentwo$. We see that $\phi$ exchanges the superclasses in Equations (\ref{Achoice1}) or (\ref{Achoice2}), so both yield supercharacter theories for $G$.

As in Case 3, we are now ready to consider the sufficient direction. Let $B$ be a supercharacter theory for $\langle \cpgen \rangle$ and let $B'$ be a refinement of $B$ such that each superclass not equal to $\{e\}$ and each supercharacter not equal to $\operatorname{triv}$ of $B$ is partitioned into 3 equal sized superclasses and supercharacters respectively in $B'$. We partition $G$ and $\irr(G)$ so that the partition of $\langle \cpgen \rangle$ and $\irr(\langle \cpgen \rangle)$ matches that of $B$, $\{ \kleingenone,\kleingentwo,\kleingenthree \}$ is an element of the partition of $G$, and $\{ \kleingenonechar,\kleingentwochar,\kleingenthreechar \}$ is an element of the partition of $\irr(G)$. As before, let $T_{\mathcal{K}}$ be the subset of the partition of $G$ disjoint from $\langle \cpgen \rangle$ and $\langle \kleingenone,\kleingentwo \rangle$, and let $T_{\mathcal{X}}$ be the set of potential supercharacters which are not summands of $\widehat{\irr(\langle \cpgenchar \rangle)}$ or $\widehat{\irr(\langle \kleingenonechar,\kleingentwochar \rangle)}$. We assume that for all $K\in T_{\mathcal{K}}$, $\Mset{K}{\kleingenone}$, $\Mset{K}{\kleingentwo}$, and $\Mset{K}{\kleingenthree}$ are pairwise disjoint, and for all $\theta \in T_{\mathcal{X}}$, $\Mset{\theta}{\kleingenonechar}$, $\Mset{\theta}{\kleingentwochar}$, and $\Mset{\theta}{\kleingenthreechar}$ are pairwise disjoint. Further, we assume that each of $\left( \{ \Mset{\theta}{\kleingenonechar}\}_{\theta\in T_{\mathcal{X}}}, \{ \Mset{K}{\kleingenone}\}_{K \in T_{\mathcal{K}}}\right)$, $\left( \{ \Mset{\theta}{\kleingentwochar}\}_{\theta\in T_{\mathcal{X}}}, \{ \Mset{K}{\kleingentwo}\}_{K \in T_{\mathcal{K}}} \right)$, and $\left( \{ \Mset{\theta}{\kleingenthreechar}\}_{\theta\in T_{\mathcal{X}}}, \{ \Mset{K}{\kleingenthree}\}_{K \in T_{\mathcal{K}}} \right)$, after adjoining $\{\operatorname{triv}\}$ and $\{e\}$, match the supercharacter theory $B'$ . For all $K\in T_{\mathcal{K}}$ and $\theta \in T_{\mathcal{X}}$ we assume that $\Mset{K}{\kleingenone}\cup \Mset{K}{\kleingentwo} \cup \Mset{K}{\kleingenthree}$ is a superclass of $B$ and $\Mset{\theta}{\kleingenonechar}\cup \Mset{\theta}{\kleingentwochar} \cup \Mset{\theta}{\kleingenthreechar}\in \mathcal{X}_{B}$. Further for every $K\in T_{\mathcal{K}}$, $\theta\in T_{\mathcal{X}}$, we suppose that for $\cpgen^k \in \Mset{K}{\kleingenone}$, $\cpgen^{\ell}\in \Mset{K}{\kleingentwo}$, and $\cpgen^r \in \Mset{K}{\kleingenthree}$ the following holds:
\begin{equation}\label{sufeq1}
\Msum{\theta}{\kleingenonechar}(\cpgen^r)=\Msum{\theta}{\kleingentwochar}(\cpgen^k)=\Msum{\theta}{\kleingenthreechar}(\cpgen^{\ell}),
\end{equation}
\begin{equation}\label{sufeq2}
\Msum{\theta}{\kleingenonechar}(\cpgen^k)=\Msum{\theta}{\kleingentwochar}(\cpgen^{\ell})=\Msum{\theta}{\kleingenthreechar}(\cpgen^r),
\end{equation}
\begin{equation}\label{sufeq3}
\Msum{\theta}{\kleingenonechar}(\cpgen^{\ell})=\Msum{\theta}{\kleingentwochar}(\cpgen^r)=\Msum{\theta}{\kleingenthreechar}(\cpgen^k).
\end{equation}

We want to show that the partition $\mathcal{K}_{B}\cup T_{\mathcal{K}} \cup \{\{ \kleingenone,\kleingentwo,\kleingenthree\}\}$ is a set of superclasses for a supercharacter theory of $G$ and that the corresponding supercharacters are those from $B$, $T_{\mathcal{X}}$, and $\kleingenonechar+\kleingentwochar+\kleingenthreechar$. Clearly the partitions of $G$ and $\irr(G)$ are the same size. Hence we only need to show that every potential supercharacter is constant on every potential superclass. Let $\theta_1$ be a supercharacter for $B$, $\theta_1 \not= \operatorname{triv}$, and let $\theta_2\in T_{\mathcal{X}}$.

Let $K\in T_{\mathcal{K}}$. Then $\widehat{K}=\Msum{K}{\kleingenone}\kleingenone+\Msum{K}{\kleingentwo}\kleingentwo +\Msum{K}{\kleingenthree}\kleingenthree$. Clearly $\kleingenonechar+\kleingentwochar+\kleingenthreechar$ is a constant $-1$ on $K$. $\theta_1$ is a supercharacter for $B$, so it is constant on $\Mset{K}{\kleingenone}\cup \Mset{K}{\kleingentwo} \cup \Mset{K}{\kleingenthree}\in \mathcal{K}_{B}$, hence $\theta_1$ is constant on $K$. Let $\cpgen^k,\cpgen^{\ell},\cpgen^r$ be as above, so $\cpgen^k\kleingenone$, $\cpgen^{\ell}\kleingentwo$, and $\cpgen^r\kleingenthree$ are elements of $K$. Then using Equations (\ref{sufeq1}), (\ref{sufeq2}), and (\ref{sufeq3}) we have
\begin{align}\notag
&\theta_2(\cpgen^k\kleingenone)=\Msum{\theta_2}{\kleingenonechar}(\cpgen^k)-\Msum{\theta_2}{\kleingentwochar}(\cpgen^k)-\Msum{\theta_2}{\kleingenthreechar}(\cpgen^k)=
\\
\notag
&\Msum{\theta_2}{\kleingentwochar}(\cpgen^{\ell})-\Msum{\theta_2}{\kleingenthreechar}(\cpgen^{\ell})-\Msum{\theta_2}{\kleingenonechar}(\cpgen^{\ell})=
\\
\notag
&\theta_2(\cpgen^{\ell}\kleingentwo)=
\\
\notag
&\Msum{\theta_2}{\kleingenthreechar}(\cpgen^r)-\Msum{\theta_2}{\kleingenonechar}(\cpgen^r)-\Msum{\theta_2}{\kleingentwochar}(\cpgen^r)=
\\
&\theta_2(\cpgen^r\kleingenthree).
\end{align}
Now suppose that for $x\in \{\kleingenone,\kleingentwo,\kleingenthree\}$ we have $x\cpgen^{\ell},x\cpgen^{\ell'}\in K$. Since $\Mset{K}{x}$ is a superclass for $B'$ and $\Msum{\theta_2}{\kleingenonechar}$, $\Msum{\theta_2}{\kleingentwochar}$, and $\Msum{\theta_2}{\kleingenthreechar}$ are supercharacters for $B'$ we have $\theta_2(x\cpgen^{\ell})=\theta_2(x\cpgen^{\ell'})$. Hence $\theta_2$ is constant on $K$.

Now suppose $K\in \mathcal{K}_{B}$. $\kleingenonechar+\kleingentwochar+\kleingenthreechar$ takes the constant value 3 on $K$. $\theta_1$ is constant on $K$ because $B$ is a supercharacter theory for $C_p$. If $\cpgen^i \in K$, then $\theta_2(\cpgen^i)$ equals
\begin{equation}\notag
\Msum{\theta_2}{\kleingenonechar}(\cpgen^i)+\Msum{\theta_2}{\kleingentwochar}(\cpgen^i)+\Msum{\theta_2}{\kleingenthreechar}(\cpgen^i)=
\end{equation}
\begin{equation}
(\Msum{\theta_2}{\kleingenonechar}+\Msum{\theta_2}{\kleingentwochar}+\Msum{\theta_2}{\kleingenthreechar})(\cpgen^i).
\end{equation}
Since $\Msum{\theta_2}{\kleingenonechar}+\Msum{\theta_2}{\kleingentwochar}+\Msum{\theta_2}{\kleingenthreechar}$ is a supercharacter for $B$, we see that $\theta_2$ is constant on $K$.

Finally, let $K=\{\kleingenone,\kleingentwo,\kleingenthree\}$. $\kleingenonechar+\kleingentwochar+\kleingenthreechar$ takes the constant value of $-1$ on $K$. $\theta_1$ is constant on $K$, taking the value equal to the number of irreducible characters which are summands of $\theta_1$ which may be expressed as $|\superclass{\chi}|$ if $\chi$ is a summand of $\theta_1$. We have
\begin{align}
\theta_2(\kleingenone) &= |\Mset{\theta_2}{\kleingenonechar}|-|\Mset{\theta_2}{\kleingentwochar}|-|\Mset{\theta_2}{\kleingenthreechar}|,\\
\theta_2(\kleingentwo) &= -|\Mset{\theta_2}{\kleingenonechar}|+|\Mset{\theta_2}{\kleingentwochar}|-|\Mset{\theta_2}{\kleingenthreechar}|,\\
\theta_2(\kleingenthree) &= -|\Mset{\theta_2}{\kleingenonechar}|-|\Mset{\theta_2}{\kleingentwochar}|+|\Mset{\theta_2}{\kleingenthreechar}|.
\end{align}
Since $|\Mset{\theta_2}{\kleingenonechar}|=|\Mset{\theta_2}{\kleingentwochar}|=|\Mset{\theta_2}{\kleingenthreechar}|$, we see that $\theta_2$ is constant on $K$. Hence we conclude that all potential supercharacters are constant on all potential superclasses, so the partition $\mathcal{K}_{B}\cup T_{\mathcal{K}} \cup \{\{ \kleingenone,\kleingentwo,\kleingenthree\}\}$ is indeed a set of superclasses for a supercharacter theory of $C_p\times C_2\times C_2$.

As in Case 3 we have presented a direct proof for the sufficient direction. However, such a supercharacter theory can be generated by a group of automorphisms in a similar fashion to Case 3. Let $A$ be a supercharacter theory satisfying the conditions of the sufficient direction above, which is not a direct product supercharacter theory. There exists a $\phi_0 \in \aut(\langle \cpgen\rangle)$ such that $\langle \phi_0 \rangle$ generates the supercharacter theory $A|_{\langle \cpgen\rangle}$. Let $\phi \in \aut(G)$ be defined by $\phi(\cpgen)=\phi_0(\cpgen)$ and $\phi|_{\langle \kleingenone,\kleingentwo\rangle}=id$. Let $\psi \in \aut(G)$ be given by $\psi(\cpgen)=\cpgen$, $\psi(\kleingenone)=\kleingentwo$, $\psi(\kleingentwo)=\kleingenthree$, and $\psi(\kleingenthree)=\kleingenone$. Then by an analogous argument as in Case 3, we see that $\langle \phi \circ \psi \rangle$ generates $A$.

For supercharacter theories generated by automorphisms as above, see Examples \ref{case41} and \ref{case42}.

\section{Case $(C_2)^3$}\label{p2}

To complete our proof, we will now discuss the supercharacter theories of $C_2\times C_2 \times C_2=\langle x \rangle \times \langle y \rangle \times \langle z \rangle$. We first note that there are 5 supercharacter theories for $C_2\times C_2$: the minimal supercharacter theory, the maximal supercharacter theory, and 3 isomorphic to the supercharacter theory with superclasses
\begin{equation}
\{e\},\{x\},\{y,xy\}.
\end{equation}
Let $A$ be a nontrivial supercharacter theory for $(C_2)^3$. For this group we cannot use Theorem \ref{Wielandt} to show that there must exist a proper nontrivial subgroup $H$ such that $\widehat{H} \in A$. So we begin by supposing there is no such $H$. Because $A$ is nontrivial, no superclass has size 7. Since every element other than the identity has order 2, the only superclass which can have size 1 is $\{e\}$. By considering the complement, this implies that no superclass has size 6. We have for $\{u,v\}$ a superclass
\begin{equation}
(u+v)^2=2e +2uv \in A.
\end{equation}
This implies that $\{uv\}$ is a superclass, so we cannot have a superclass of size 2. Hence we cannot have a superclass of size 5, again by considering complements. This leaves only a superclass of size 3 and a superclass of size 4 as a valid option. If $\{u,v,w\}$ is the superclass of size 3, then by assumption $\{e,u,v,w\}$ is not a subgroup. Therefore $\{u,v,w\}$ is disjoint from $\{uv,uw,vw\}$. We have
\begin{equation}
(u+v+w)^2=3e+2(uv+uw+vw) \in A.
\end{equation}

 Hence $\{uv,uw,vw\}$ is a union of superclasses and is different from $\{u,v,w\}$ which is a contradiction. Therefore we see that for every nontrivial supercharacter theory of $(C_2)^3$ there exists a proper nontrivial subgroup $H$ such that $\widehat{H} \in A$.

There are only two isomorphism types for $H$, either $H\cong C_2 \times C_2$ or $H\cong C_2$. Suppose that there is at least one $H\cong C_2 \times C_2$ such that $\widehat{H} \in A$, and without loss of generality let $H=\langle x, y \rangle$. We consider the possible partitions of the complement of $H$, $\{z,xz,yz,xyz\}$. If $\{z,xz,yz,xyz\}$ is a superclass then the supercharacter theory is a wedge product of the supercharacter theory of $H$ and the unique supercharacter theory of $(C_2)^3/H\cong C_2$. Such a supercharacter theory of $(C_2)^3$ exists for every supercharacter theory of $H$, and the set of superclasses is 
\begin{equation}
\mathcal{K}_{A|_H}\cup\{\{z,xz,yz,xyz\}\}.
\end{equation}
If the partition of $\{z,xz,yz,xyz\}$ contains a singleton, then there exists a subgroup $H'\cong C_2$ with $\widehat{H'}\in A$ and $H \cap H' =\{e\}$. By Lemma \ref{refine} we have that $A$ must be the direct product of the supercharacter theories of $H$ and $H'$. The only remaining possible partition of $\{z,xz,yz,xyz\}$ is two pairs, let them be $\{t,u\}$ and $\{v,w\}$. We note that $tu=vw\in H$, without loss of generality let $tu=vw=x$. Then $A|_{H}$ is either the minimal supercharacter theory of $H$ or $\{\{e\},\{x\},\{y,z\}\}$. If it is the minimal supercharacter theory of $H$, then it is in fact a direct product supercharacter theory of $\langle t,u\rangle$ and $\langle y \rangle$, so we have already accounted for this situation. If the supercharacter theory of $H$ is $\{\{e\},\{x\},\{y,z\}\}$, then we see that every superclass disjoint from $\langle x \rangle$ is a $\langle x \rangle$-coset,  so $A$ is the wedge product of the unique supercharacter theory of $\langle x \rangle$ and the minimal supercharacter theory of $(C_2)^3/\langle x \rangle\cong C_2 \times C_2$.

We now consider the case where there does not exist an $H\cong C_2 \times C_2$ with $\widehat{H}\in A$. There must exist $H' \cong C_2$ with $\widehat{H'} \in A$. If there is a singleton superclass $\{u\}$ different from $\{e\}$ and $H' \not=\langle u \rangle$, then $\widehat{H'\times \langle u \rangle} \cong C_2\times C_2$ and $\widehat{H'\times \langle u \rangle} \in A$ which contradicts our assumption. Hence $H'$ is the only proper nontrivial subgroup which is a union of superclasses. We determine how many possibilities there are for $A$ by considering the dual supercharacter theory of $A$, which we will denote by $B$. There is exactly one proper nontrivial subgroup $H$ such that $\widehat{H} \in B$, and $H\cong C_2\times C_2$. Hence $B|_{H}$ is the maximal supercharacter theory, and by the above argument, the complement of $H$ is a superclass. Hence there is only one isomorphism type for $B$ and so there is only one isomorphism type for $A$. $A$ must be the wedge product of the supercharacter theory of $H'$ and the maximal supercharacter theory of $(C_2)^3/ H' \cong C_2\times C_2$.

This completes the proof of Theorem \ref{sufficient}, as we have considered every case.

\section{Enumeration of the Supercharacter Theories of $C_p \times C_2 \times C_2$ for $p$ odd}

We will now determine the number of supercharacter theories of $C_p\times C_2 \times C_2$ when $p$ is odd, and how many can be obtained by each of the three constructions. Working through the argument provides a method for constructing the complete set of supercharacter theories for a particular choice of $p$. 

We will let $\divisor(n)$ equal the number of positive integers that divide $n$.

\begin{lemma}\label{directauto}
Let $G$ be a finite group with a complementary pair $H$, $N$, and let $A$ be a supercharacter theory for $G$ such that $A$ is the direct product supercharacter theory of $A|_{H}$ and $A|_{N}$. Then $A$ is generated by automorphisms iff $A|_{H}$ and $A|_{N}$ are both generated by automorphisms.
\end{lemma}
\begin{proof}
Since $H$ and $N$ are a complementary pair, $G\cong H\times N$. Clearly if $A$ is generated by automorphisms and $K$ is a normal subgroup such that $\widehat{K}\in A$, then $A|_{K}$ is generated by automorphisms. Hence $A|_{H}$ and $A|_{N}$ are both generated by automorphisms by restriction. Conversely, suppose that there exists $H' \leq \aut(H)$ and $N' \leq \aut(N)$ such that $H'$ and $N'$ generate the supercharacter theories $A|_{H}$ and $A|_{N}$ respectively. $H'$ includes into $\aut(G)$ by acting trivially on $N\leq G$, and similarly $N'$ includes into $\aut(G)$ by acting trivially on $H \leq G$. Then the subgroup of $\aut(G)$ generated by the images of $H'$ and $N'$ generates $A$. 
\end{proof}

\begin{lemma}\label{CpC2}
If $p$ is an odd prime, then there are $3\divisor(p-1)+1$ supercharacter theories of $C_p \times C_2$. Further a supercharacter theory of $C_p \times C_2$ is generated by automorphisms iff it is a direct product supercharacter theory.
\end{lemma}
\begin{proof}
Let $C_p \times C_2=\langle \cpgen \rangle \times \langle \kleingenone \rangle$. Since $C_p \times C_2\cong C_{2p}$ is cyclic, by \cite[Th. 3.7]{cyclicSchurII} every nonmaximal supercharacter theory can be constructed as a wedge product, a direct product, or is generated by automorphisms. Since $p$ and 2 are relatively prime, it is clear that if $A$ is generated by automorphisms, $\widehat{\langle \cpgen \rangle} \in A$ and $\widehat{\langle \kleingenone \rangle} \in A$. By Lemma \ref{refine} $A$ must be a refinement of the direct product supercharacter theory of $A|_{\langle \cpgen \rangle}$ and $A|_{\langle \kleingenone \rangle}$. Since $\kleingenone \in A$, we see that $A$ must be the direct product supercharacter theory. Conversely, if $A$ is a direct product supercharacter theory, then $\widehat{\langle \cpgen \rangle} \in A$ and $\widehat{\langle \kleingenone \rangle} \in A$. By Lemma \ref{fromauto} every supercharacter theory of $\langle \cpgen \rangle$ and $\langle \kleingenone \rangle$ is generated by automorphisms. Therefore by Lemma \ref{directauto} $A$ is generated by automorphisms.

The direct product supercharacter theories are in bijection with the supercharacter theories of $\langle \cpgen \rangle$, and by Lemma \ref{fromauto} there are $\divisor(p-1)$ of them. There are $2\divisor(p-1)$ wedge products, $\divisor(p-1)$ where $\langle \cpgen \rangle$ is the normal group and $\divisor(p-1)$ where $\langle \cpgen \rangle$ is the quotient group. The only other supercharacter theory is the maximal supercharacter theory, hence $C_p \times C_2$ has $3\divisor(p-1)+1$ supercharacter theories.
\end{proof}

\begin{lemma}\label{wedgeinclusion}
\cite[Lemma 8.1]{Hendconstruction}
Let $N$ and $H$ be proper nontrivial normal subgroups of $G$. If a supercharacter theory of $G$ can be constructed as a wedge product of supercharacter theories for $N$ and $G/N$ and also as a wedge product of supercharacter theories of $H$ and $G/H$ then $N\leq H$ or $H \leq N$.
\end{lemma}

\begin{theorem}
Let $p$ be an odd prime and let $p-1=2^k3^{\ell}n$ where $n$ is not divisible by 2 or 3. Then  the number of distinct supercharacter theories of $C_p\times C_2 \times C_2$ is
\begin{equation}
3k\divisor(3^{\ell}n)+2\ell \divisor(2^k n) + 30\divisor(p-1)+13.
\end{equation}
Further, we can enumerate the supercharacter theories by method of construction. There are $3k\divisor(3^{\ell}n)+2\ell \divisor(2^k n)+5\divisor(p-1)$ supercharacter theories which can be generated by automorphisms, and $11\divisor(p-1)+6$ supercharacter theories which are direct products, including $5\divisor(p-1)$ which can also be generated by automorphisms. If a supercharacter theory is a wedge product, then it is not a direct product or generated by automorphisms. There are $19\divisor(p-1)+6$ supercharacter theories which are wedge products, and the only remaining supercharacter theory is the maximal supercharacter theory.
\end{theorem}
\begin{proof}
By Theorem \ref{sufficient}, every nonmaximal supercharacter theory of $C_p\times C_2 \times C_2$ is generated by automorphisms, a direct product, or a wedge product. We begin by considering all of the supercharacter theories generated by automorphisms. Every subgroup of $\aut(C_p\times C_2 \times C_2)$ generates a supercharacter theory of $C_p\times C_2 \times C_2$, although they are not all distinct. Also note that the minimal supercharacter theory is generated by the trivial subgroup of $\aut(C_p\times C_2 \times C_2)$. Since $\aut(C_p)\cong C_{p-1}$, $\aut(C_2\times C_2)\cong S_3$, and $p$ is odd we have that
\begin{equation}
\aut(C_p\times C_2 \times C_2)\cong \aut(C_p)\times \aut(C_2\times C_2) \cong C_{p-1} \times S_3.
\end{equation}
We will now consider all possible subgroups of $\aut(C_p) \times \aut(C_2\times C_2)$. Let $\aut(C_p)=\langle \psi \rangle$, let $A_3$ be the order three subgroup of $\aut(C_2\times C_2)$ and for $x \in \{\kleingenone,\kleingentwo,\kleingenthree\}$ let $H_x$ be the order two subgroup of $\aut(C_2\times C_2)$ that fixes $x$. It follows from Goursat's Lemma that every subgroup of $G_1\times G_2$ can be expressed as
\begin{equation}\label{Goursat}
R=\{(x,y)\in H_1\times H_2|\phi(xN_1)=yN_2\}
\end{equation}
for a unique choice of $N_1 \trianglelefteq H_1 \leq G_1$, $N_2 \trianglelefteq H_2 \leq G_2$ and an isomorphism $\phi:H_1/N_1 \rightarrow H_2/N_2$, and every such choice gives a subgroup.
Let $N_1 \trianglelefteq H_1 \leq \aut(C_p)$ and $N_2 \trianglelefteq H_2 \leq \aut(C_2\times C_2)$. We have $\aut(C_p)\cong C_{2^k}\times C_{3^\ell}\times C_n$ where $n$ is not divisible by 2 or 3.

$H_2/N_2$ is isomorphic to the trivial group, $C_2$, or $C_3$. Case I: If $H_2/N_2 \cong \{e\}$, then $H_1=N_1$ and $H_2=N_2$. There is only one choice for $\phi$, so it is easy to see that in this case the subgroup $R$ is the direct product $H_1\times H_2$. Every choice of $H_1$ gives a distinct supercharacter theory of $\langle \cpgen \rangle$, so there are $\divisor(p-1)$ choices for $H_1$. However, $\aut(C_2\times C_2)$ and $A_3 \leq \aut(C_2\times C_2)$ both generate the maximal supercharacter theory of $C_2\times C_2$. We then observe that there are five choices for $H_2$. Hence there are $5\divisor(p-1)$ possibilities. 

Case II: There are four ways for $H_2/N_2 \cong C_2$: $H_2=\aut(C_2\times C_2)$ and $N_2=A_3$, or $H_2$ is $H_\kleingenone$, $H_\kleingentwo$, or $H_\kleingenthree$ and $N_2=\{e\}$.  We have $H_1=\langle \psi^m \rangle$ and $N_1=\langle \psi^{2m} \rangle$ where $m|p-1$ and $(p-1)/m$ is even. There is only one choice for $\phi$. If $H_2=\aut(C_2\times C_2)$ and $N_2=A_3$ then $(id,\sigma)$ and $(id,\sigma^2)$ are elements of the subgroup $R$ where $\sigma(\kleingenone)=\kleingentwo$, $\sigma(\kleingentwo)=\kleingenthree$, $\sigma(\kleingenthree)=\kleingenone$. We see that the supercharacter theory generated by this subgroup $R$ is the same one generated by the subgroup for $H'_1=N'_1=H_1$, $H'_2=N'_2=A_3$, which is $\langle \psi^m \rangle \times \langle \sigma \rangle$. Hence this supercharacter theory has already been constructed above. 

So we will only consider $H_2=H_x$ and $N_2 =\{id\}$. There is only one choice for $\phi$ so the resulting subgroup $R$ is $\langle \psi^m \circ \tau \rangle\cong C_{(p-1)/m}$, where $H_x=\langle \tau \rangle$. There are three choices for $H_2$, $k\divisor(3^{\ell}n)$ choices for $H_1$, and one choice for $N_1$, which gives $3k\divisor(3^{\ell}n)$ new contributions. Note that supercharacter theories of this form correspond to the supercharacter theories in Case 3 of Theorem \ref{sufficient} which are not direct products. See also Example \ref{exwithM}.

Case III: When $H_2/N_2\cong C_3$, we must have $H_2=A_3$ and $N_2=\{e\}$. Then $H_1 =\langle \psi ^m \rangle$, $N_1=\langle \psi^{3m} \rangle$ where $m|(p-1)$ and 3 divides $(p-1)/m$. There are $\ell\divisor(2^k n)$ possibilities for $H_1$, a unique $N_1$ and two choices for $\phi$ which yields $2\ell \divisor(2^k n)$ subgroups $R=\langle \psi^m \circ \sigma_\phi \rangle \cong C_{(p-1)/m}$ where $\phi(\psi^{m}N_1)=\sigma_\phi$. Note that supercharacter theories of this form correspond to the supercharacter theories in Case 4 of Theorem \ref{sufficient} which are not direct products. See Examples \ref{case41}, \ref{case42}. We conclude that there are $3k\divisor(3^{\ell}n)+2\ell \divisor(2^k n)+5\divisor(p-1)$ supercharacter theories generated by automorphisms.

We now consider the supercharacter theories which are direct products. By Lemma \ref{directauto}, we see that some direct product supercharacter theories of $C_p\times C_2 \times C_2$ can be generated by automorphisms. In particular, by Lemma \ref{fromauto} every supercharacter theory of $C_p$ is generated by automorphisms, and it is easy to see that every supercharacter theory of $C_2 \times C_2$ is also generated by automorphisms. Hence we have already constructed the $5\divisor(p-1)$ direct product supercharacter theories of $C_p\times C_2 \times C_2$ with the complementary pair $C_p$, $C_2 \times C_2$. The other complementary pair to be considered is $C_p \times C_2$, $C_2$. Let $x,y$ be distinct elements of $\{\kleingenone, \kleingentwo, \kleingenthree\}$ and let $C_p \times C_2 = \langle \cpgen, x \rangle$, and $C_2=\langle y \rangle$. Note that if a supercharacter theory $A$ can be expressed as the direct product for the complementry pair $ \langle \cpgen, x \rangle$, $\langle y \rangle$ then it is also a direct product for the complementary pair $\langle \cpgen \rangle$, $\langle x,y \rangle$ iff $A|_{\langle \cpgen, x \rangle}$ is a direct product supercharacter theory. Therefore we want $A|_{\langle \cpgen, x \rangle}$ to not be a direct product. Hence we need to count the supercharacter theories $A$ which are direct products of $A|_{\langle \cpgen, x \rangle}$ and $A|_{\langle y \rangle}$ such that $A|_{\langle \cpgen, x \rangle}$ must be a wedge product with $\langle \cpgen \rangle$ normal, a wedge product with $\langle x \rangle$ normal, or the maximal supercharacter theory. There are six choices for the pair $x,y$. In the case of $\langle \cpgen \rangle$ normal there are $\divisor(p-1)$ choices for the supercharacter theory $A|_{\langle \cpgen \rangle}$, and given $y$ either choice of $x$ yields the same supercharacter theory so there are $3\divisor(p-1)$ possibilities. Similarly, for $\langle x \rangle$ normal there are $\divisor(p-1)$ supercharacter theories for the quotient, and either choice of $y$ yields the same supercharacter theory so there are $3\divisor(p-1)$ possibilities. For the maximal case, all six choices of $x,y$ give distinct supercharacter theories. Hence there are $6\divisor(p-1)+6$ direct product supercharacter theories of $C_p\times C_2 \times C_2$ which cannot be generated by automorphisms.

We now consider the wedge products. First note that if $A$ is a wedge product, there does not exist a complementary pair of subgroups such that both are unions of superclasses of $A$. Hence $A$ cannot be either a direct product supercharacter theory or generated by automorphisms. If $A$ is a wedge product of the supercharacter theories of a normal subgroup $N$ and $G/N$, then $N$ is isomorphic to one of  $C_2$, $C_2\times C_2$, $C_p$, or $C_2 \times C_p$. We will avoid constructing a supercharacter theory more than once by using Lemma \ref{wedgeinclusion}. We begin with all wedge products with $N\cong C_p$. There is only one subgroup isomorphic to $C_p$, so $N = \langle \cpgen \rangle$. There are $\divisor(p-1)$ supercharacter theories of $N$, and five supercharacter theories of the quotient, hence there are $5\divisor(p-1)$ such supercharacter theories.

Now we consider all wedge products with $N\cong C_2$. There are three possibilities for $N$: $\langle \kleingenone \rangle$, $\langle \kleingentwo \rangle$, and $\langle \kleingenthree \rangle$. $N$ has only 1 supercharacter theory, and by Lemma \ref{CpC2} the quotient has $3\divisor(p-1)+1$ supercharacter theories. Hence there are $3(3\divisor(p-1)+1)$ wedge products with $N\cong C_2$ normal.

We now want to count the wedge products with $N\cong C_2 \times C_2$ which are not also wedge products of supercharacter theories of $N'\cong C_2$ and $G/N'$. We have $N=\{e,\kleingenone,\kleingentwo,\kleingenthree\}$. It is easy to check that a supercharacter theory $A$ of $C_p \times C_2 \times C_2$ is a wedge product for both $N$ and $N'$ iff the dimension of $A|_{N}$ is three. Hence $A|_{N}$ must be either the minimal supercharacter theory, or the maximal supercharacter theory. By Lemma \ref{fromauto} there are $\divisor(p-1)$ supercharacter theories for the quotient, hence we have $2\divisor(p-1)$ supercharacter theories.

Finally, we count the wedge products where $N\cong C_p\times C_2$ which are not wedge products for $N'\cong C_p$ or $N'' \cong C_2$. If $A$ is a wedge product for $C_p$ or $C_2$ then $A|_N$ is a wedge product for $C_p$ or $C_2$ respectively. Hence $A|_N$ must be either a direct product supercharacter theory or the maximal supercharacter theory, so there are $\divisor(p-1)+1$ choices for $A|_{N}$. There is only one choice for the supercharacter theory of the quotient, and three choices for $N$: $\langle \cpgen, \kleingenone \rangle$, $\langle \cpgen, \kleingentwo \rangle$, or $\langle \cpgen, \kleingenthree \rangle$. Hence there are $3(\divisor(p-1)+1)$ such supercharacter theories.

We conclude that there are $19\divisor(p-1) +6$ wedge products. Adding to this the supercharacter theories generated by automorphisms, the direct product supercharacter theories which are not generated by automorphisms, and the maximal supercharacter theory we see that there are $3k\divisor(3^{\ell}n)+2\ell \divisor(2^k n) + 30\divisor(p-1)+13$ supercharacter theories for $C_p \times C_2 \times C_2$.

\end{proof}

\section{Examples}

\begin{example}\label{exwithM}
Let $p=5$. Let $\phi,\psi$ be automorphisms of $C_5\times C_2\times C_2$ defined by $\phi(\cpgen)=\cpgen^2$, $\phi(\kleingenone)=\kleingenone$, $\phi(\kleingentwo)=\kleingentwo$, $\phi(\kleingenthree)=\kleingenthree$, and $\psi(\cpgen)=\cpgen$, $\psi(\kleingenone)=\kleingenone$, $\psi(\kleingentwo)=\kleingenthree$, $\psi(\kleingenthree)=\kleingentwo$. Then the supercharacter theory generated by the group $\langle \phi \circ \psi \rangle$ has the following superclasses:
\begin{equation}
\{e\}, \{\cpgen,\cpgen^2,\cpgen^3,\cpgen^4 \},
\end{equation}
\begin{equation}
\{\kleingenone\},\{\kleingentwo,\kleingenthree\},
\end{equation}
\begin{equation}
\{\cpgen\kleingenone,\cpgen^2 \kleingenone,\cpgen^3 \kleingenone, \cpgen^4 \kleingenone\},
\end{equation}
\begin{equation}
\{\cpgen \kleingentwo, \cpgen^4 \kleingentwo, \cpgen^2 \kleingenthree, \cpgen^3 \kleingenthree\}, \{\cpgen^2 \kleingentwo, \cpgen^3 \kleingentwo, \cpgen \kleingenthree, \cpgen^4 \kleingenthree \}.
\end{equation}
For $K=\superclass{\cpgen \kleingentwo}$ we have $\Mset{K}{\kleingentwo}=\{\cpgen,\cpgen^4\}$ and $\Mset{K}{\kleingenthree}=\{\cpgen^2,\cpgen^3\}$. We see that $\{e\}, \{\cpgen,\cpgen^4\}, \{\cpgen^2,\cpgen^3\}$ are the superclasses of the supercharacter theory of $\langle \cpgen \rangle$ generated by the group $\langle \phi^2|_{\langle \cpgen \rangle} \rangle$ and that $\phi^2(\cpgen)=\cpgen^{-1}$.
\end{example}

\begin{example}\label{case41}
Let $p=7$. $\{ \{e\},\{\cpgen,\cpgen^2,\cpgen^4\},\{\cpgen^3,\cpgen^5,\cpgen^6\}\}$ is the set of superclasses for the supercharacter theory of $C_7$ generated by $\langle\phi_0 \rangle$ where $\phi_0(\cpgen)=\cpgen^2$. Let $\phi,\psi \in \aut(C_7\times C_2\times C_2)$ be defined by $\phi(\cpgen)=\cpgen^2$, $\phi(\kleingenone)=\kleingenone$, and $\phi(\kleingentwo)=\kleingentwo$, $\phi(\kleingenthree)=\kleingenthree$, and $\psi(\cpgen)=\cpgen$, $\psi(\kleingenone)=\kleingentwo$, $\psi(\kleingentwo)=\kleingenthree$, $\psi(\kleingenthree)=\kleingenone$. Then $\langle \phi\circ \psi \rangle$ generates the supercharacter theory with superclasses:
\begin{equation}
\{e\}, \{\cpgen,\cpgen^2,\cpgen^4\},\{\cpgen^3,\cpgen^5,\cpgen^6\},
\end{equation}
\begin{equation}
\{\kleingenone,\kleingentwo,\kleingenthree\},
\end{equation}
\begin{equation}
\{ \cpgen \kleingenone,\cpgen^2 \kleingentwo, \cpgen^4 \kleingenthree\}, \{\cpgen^2 \kleingenone, \cpgen^4 \kleingentwo, \cpgen \kleingenthree\}, \{ \cpgen^6 \kleingenone, \cpgen^5 \kleingentwo, \cpgen^3 \kleingenthree\},
\end{equation}
\begin{equation}
\{\cpgen^5 \kleingenone, \cpgen^3 \kleingentwo, \cpgen^6 \kleingenthree\}, \{ \cpgen^3 \kleingenone, \cpgen^6 \kleingentwo, \cpgen^5 \kleingenthree\}, \{\cpgen^4 \kleingenone, \cpgen \kleingentwo, \cpgen^2 \kleingenthree\}.
\end{equation}

Note that the supercharacter theory generated by $\langle \phi\circ \psi^{-1} \rangle$ is
\begin{equation}
\{e\}, \{\cpgen,\cpgen^2,\cpgen^4\},\{\cpgen^3,\cpgen^5,\cpgen^6\},
\end{equation}
\begin{equation}
\{\kleingenone,\kleingentwo,\kleingenthree\},
\end{equation}
\begin{equation}
\{ \cpgen \kleingenone, \cpgen^2 \kleingenthree,\cpgen^4 \kleingentwo\}, \{\cpgen^2 \kleingenone, \cpgen^4 \kleingenthree, \cpgen \kleingentwo\}, \{ \cpgen^6 \kleingenone, \cpgen^5 \kleingenthree, \cpgen^3 \kleingentwo\},
\end{equation}
\begin{equation}
\{\cpgen^5 \kleingenone, \cpgen^3 \kleingenthree, \cpgen^6 \kleingentwo\}, \{ \cpgen^3 \kleingenone, \cpgen^6 \kleingenthree, \cpgen^5 \kleingentwo\}, \{\cpgen^4 \kleingenone, \cpgen \kleingenthree, \cpgen^2 \kleingentwo\}.
\end{equation}
These supercharacter theories differ according to the 2 possibilities described in Equations (\ref{Achoice1}) and (\ref{Achoice2}), and also the different choices for the isomorphism in Equation (\ref{Goursat}).

\end{example}
\begin{example}\label{case42}
Let $\psi$ be defined as above and let $\sigma \in \aut(C_7\times C_2\times C_2)$ be defined by $\sigma(\cpgen)=\cpgen^5$, $\sigma(\kleingenone)=\kleingenone$, $\sigma(\kleingentwo)=\kleingentwo$, and $\sigma(\kleingenthree)=\kleingenthree$. Then $\langle \sigma \circ \psi \rangle$ generates the following supercharacter theory:
\begin{equation}
\{e\}, \{\cpgen,\cpgen^2,\cpgen^3,\cpgen^4,\cpgen^5,\cpgen^6\},
\end{equation}
\begin{equation}
\{\kleingenone,\kleingentwo,\kleingenthree\},
\end{equation}
\begin{equation}
\{ \cpgen \kleingenone, \cpgen^6 \kleingenone, \cpgen^2 \kleingentwo, \cpgen^5 \kleingentwo, \cpgen^3 \kleingenthree, \cpgen^4 \kleingenthree\},
\end{equation}
\begin{equation}
\{ \cpgen^2 \kleingenone, \cpgen^5 \kleingenone, \cpgen^3 \kleingentwo, \cpgen^4 \kleingentwo, \cpgen \kleingenthree, \cpgen^6 \kleingenthree\},
\end{equation}
\begin{equation}
\{ \cpgen^3 \kleingenone, \cpgen^4 \kleingenone, \cpgen \kleingentwo,\cpgen^6 \kleingentwo, \cpgen^2 \kleingenthree, \cpgen^5 \kleingenthree\}.
\end{equation}
\end{example}

\bibliographystyle{alpha}
\bibliography{supercharref}

\textsc{Department of Mathematics, University of California, Davis, One Shields Avenue, Davis, CA} 95616-8633

\textit{E-mail address}:\texttt{ langa@math.ucdavis.edu}

\end{document}